\documentclass[12pt,a4paper]{article}
\pdfoutput=1
\usepackage[english]{babel}
\usepackage{verbatim}
\usepackage{bm, color}

\usepackage{amsmath}
\usepackage{amssymb}
\usepackage{amsfonts}
\usepackage{amsbsy}
\usepackage{amsthm}
\usepackage{graphicx}
\usepackage{graphics}
\usepackage{subfigure}
\usepackage{epsfig}

\usepackage{array}
\usepackage{booktabs}
\usepackage{caption}
\usepackage{tabularx}

\newcommand{\R}{{\mathbb R}}

\renewcommand{\P}{\mathbb P}
\newcommand{\Z}{\mathbb Z}

\newcommand{\dd}{\mathrm d}
\newcommand{\dtm}{\frac{\Delta t}{2}}
\newcommand{\dt}{{\Delta t}}

\newtheorem{Theorem}{Theorem}

\newtheorem{Prop}[Theorem]{Proposition}

\begin{document}

\title{Second order \\ fully semi-Lagrangian discretizations\\
 of advection--diffusion--reaction systems}

\author{Luca Bonaventura$^{(1)}$\ \ Elisabetta Carlini$^{(2)}$\\
  Elisa Calzola$^{(2)}$ Roberto Ferretti$^{(3)}$}

\maketitle

\begin{center}
{\small
$^{(1)}$ MOX -- Modelling and Scientific Computing, \\
Dipartimento di Matematica, Politecnico di Milano \\
Via Bonardi 9, 20133 Milano, Italy\\
{\tt luca.bonaventura@polimi.it}
}
\end{center}
\begin{center}
{\small
$^{(2)}$
Dipartimento di Matematica\\
Universit\`a degli Studi Roma Sapienza\\
 P.le Aldo Moro 5,
00185, Roma, Italy\\
 {\tt calzola@mat.uniroma1.it}, {\tt carlini@mat.uniroma1.it}
}
\end{center}

\begin{center}
{\small
$^{(3)}$
Dipartimento di Matematica e Fisica\\
Universit\`a degli Studi Roma Tre\\
 L.go S. Leonardo Murialdo 1,
00146, Roma, Italy\\
{\tt ferretti@mat.uniroma3.it}
}
\end{center}

\date{}


\vspace*{0.5cm}


\vspace*{0.5cm}

\pagebreak

\abstract{We propose a second order, fully semi-Lagrangian method for the numerical solution of systems of 
advection--diffusion--reaction equations, which employs a  semi-Lagrangian approach to approximate in time both the advective and  the diffusive terms. Standard interpolation procedures are used   
for the space discretization on structured and unstructured meshes.
The  proposed method  allows for large time steps, while avoiding the solution of large linear systems, which would be required by an implicit time discretization technique. Numerical experiments demonstrate the effectiveness of the proposed approach and its superior efficiency with respect to more conventional explicit and implicit time discretizations.}

\pagebreak
\section{Introduction}
\label{intro}

Large systems of advection--diffusion--reaction equations are responsible for most of the computational cost of typical environmental fluid dynamics models, such as  are applied in climate modelling, water and air quality models and oceanic biogeochemistry \cite{bonaventura:2012}, \cite{erath:2012}, \cite{erath:2016}. Also in short and medium range weather forecasting, the number of interacting transported   species is significant and drives the choice of optimal time discretization approaches towards methods that allow  the use of large time steps \cite{wedi:2015}.
Due to the potentially very large number of equations of this kind that have to be solved simultaneously in order  to achieve a complete description of the relevant physical processes, even minor efficiency gains
in  the  discretization of this very classical problem are of paramount practical importance. 
The standard ways to achieve such optimal efficiency are either the use of implicit schemes, or the
application of semi-Lagrangian (SL) techniques, \cite{falcone:2013}, \cite{staniforth:1991} to the advection step, coupled to implicit methods for the diffusion and reaction step. As discussed in  \cite{erath:2012}, \cite{erath:2016}, SL methods have the advantage that all the computational work that makes them computationally more expensive per time step
than standard Eulerian techniques is indeed independent of the number of tracers, which allows to
achieve easily a superior efficiency level in the limit of a large number of tracers.

In  the recent papers \cite{bonaventura:2014}, \cite{bonaventura:2018}, a fully SL approach to both
the advection and diffusion step was pursued, which combines the standard SL treatment of advection
with SL-like schemes for diffusion, proposed, among others, in  \cite{camilli:1995},   \cite{ferretti:2010}, \cite{milstein:2002}, \cite{milstein:2000}, \cite{milstein:2001}. In particular, it was shown in  \cite{bonaventura:2018}
that, even for a single advection--diffusion equation, the fully SL approach can be more efficient than
standard implicit techniques. 

In the present work, we present a number of improvements to the fully
SL  approach of \cite{bonaventura:2014}, \cite{bonaventura:2018}. In particular, we show how second order accuracy in time can be achieved. An improved  treatment of Dirichlet boundary conditions is also discussed and analysed.
The resulting approach yields  a very efficient combination, which is 
 validated on a number of classical benchmarks, both on structured and unstructured meshes.
 Numerical results  show that the method yields good quantitative agreement with reference
numerical solutions, while being superior in efficiency to standard implicit approaches and to approaches
in which the SL method is only used for the advection term. 


The outline of the paper is the following. In Section \ref{model}, the  considered model problems are introduced. Section \ref{fullysl} describes the SL advection--diffusion solver. A stability and convergence analysis of the method is outlined in Section
\ref{convergence}.
The possible approaches to the treatment of boundary conditions are discussed in Section \ref{bc}.
A numerical validation of the proposed approach on both structured and unstructured meshes 
is presented in Section \ref{tests}, while some conclusions and perspectives for future developments
are outlined in Section \ref{conclu}.

\section{The model problem}
\label{model}

We consider as a model problem the advection--diffusion--reaction equation with Dirichlet boundary conditions
\begin{eqnarray}\label{eq:adr}
&& c_t +   u \cdot \nabla  c - \nu \Delta c    = f(c) \hskip 2cm (x,t)\in\Omega \times \left( 0,T \right], \nonumber \\
&& c \left(x,t \right)= b \left(  x, t \right)   \hskip 3.2cm  (x,t)\in\partial \Omega \times \left( 0,T \right], \\
&& c \left(  x,0 \right)=c_0 \left(  x \right)  \hskip 3.5cm   x\in\Omega. \nonumber
\end{eqnarray}
The unknown $ c: \Omega\times[0,T]\to \mathbb{R} $  can be interpreted as the concentration of a chemical species that is transported through the domain $ \Omega $ by the advection and diffusion processes, while 
undergoing locally a nonlinear evolution determined by the source term $f(c).$
$T $ is the final time, $\Omega\subset \mathbb{R}^2$ is an open bounded domain, $  u: \Omega\times[0,T]\to \mathbb{R}^2 $ is a  velocity field and $ b: \partial \Omega\times[0,T]\to \mathbb{R}$  denotes the boundary value of the species $c. $ While the proposed numerical method will be presented in this simpler case, the target for more realistic applications are systems of   coupled advection--diffusion--reaction equations of the form
\begin{eqnarray}\label{eq:adrsys}
&&(c_k)_t+ u\cdot \nabla c_k - \nabla\cdot \left( D\nabla c_k \right ) = f_k(c_1,\dots,c_S) \hskip 0.3cm (x,t)\in\Omega \times (0,T], \nonumber\\
&&c_k (x,t)= b_k (x, t)   \hskip 3.4cm (x,t)\in\partial \Omega \times ( 0,T], \\
&&c_k (x,0)=c_{0,k} (x)   \hskip 3.4cm x\in\Omega, k=1,\dots,S. \nonumber
\end{eqnarray}
Here, $D$ denotes a symmetric and positive semi-definite diffusivity tensor, possibly dependent on
space and time. As remarked in Section \ref{intro},
systems of this kind, with a possibly large number of species $S$, are responsible for the largest share of the computational cost of typical environmental fluid dynamics models, so that even   minor increases
in the efficiency of the  discretization for this very classical problem are of great practical relevance.

\pagebreak

 \section{Fully semi-Lagrangian  methods}
 \label{fullysl}
 
 It can be observed that the evolution equation \eqref{eq:adr} is of the form
 $$c_t    = \mathcal L c +f(c), $$
 where ${\mathcal L}$ denotes a linear differential operator. 
 Consider then the homogeneous equation
\begin{equation}
\label{homeq}
\tilde c_t  = {\mathcal L}\tilde c 
\end{equation}
   associated to  \eqref{eq:adr}
 with the same initial datum $\tilde c(x,0)=c_0(x) $ and boundary conditions as in \eqref{eq:adr}.
 If the evolution operator determining the solution of \eqref{homeq} is denoted by ${\mathcal E}_t, $
 so that $ \tilde c(t)=  {\mathcal E}_t[c_0], $ by formal application of the variation of constants formula, the solution of 
  \eqref{eq:adr} can then be represented as
 \begin{equation} 
c(x,t)={\mathcal E}_t[c_0](x)    
+\int _0^t {\mathcal E}_{t-s}[f\circ c](x)   \ \dd s.  
\end{equation}
If discrete time levels $t_{n}, n=0,\dots, N $ are introduced, so that
$t_{n}=n\Delta t $ and $\Delta t=T/N,$ a numerical method for the solution 
of equation  \eqref{eq:adr} on the interval $[t_{n},t_{n+1}] $ can then be derived from the representation
formula
\begin{equation} 
\label{repform1}
c(x,t_{n+1})={\mathcal E}_{\Delta t}[c(.,t_{n})](x)    
+\int _{t_{n}}^{t_{n+1}} {\mathcal E}_{t_{n+1}-s}[f\circ c](x)   \ \dd s.  
\end{equation}
\noindent
Discretizing the time integral by the trapezoidal rule, one obtains
\begin{eqnarray} 
\label{repform2}
c(x,t_{n+1})&\approx &{\mathcal E}_{\Delta t}[c(.,t_{n})](x)  \nonumber  \\
&&+\frac{\Delta t}2 \left   [  {\mathcal E}_{\Delta t}[f\circ c](x)  + f(c(x,t_{n+1}) )     \right ] .
\end{eqnarray}
If the diffusion term is dropped in equation \eqref{eq:adr} and the evolution operator is approximated
by a numerical approximation of the flow streamline and an interpolation at the departure point of the
streamline, 
a numerical method based on
formula
  \eqref{repform2} can be interpreted as a semi-Lagrangian extension of the trapezoidal rule
  with global truncation error of order 2.
 Semi-Lagrangian methods based on this formula have
    been used successfully in a large number of applications,
see, among many others, 
 \cite{bonaventura:2000},\cite{casulli:1994},\cite{cote:1988},\cite{temperton:2001},\cite{temperton:1987},\cite{tumolo:2013}. Often, a first order, off-centered version of the above formula is employed,
defined for $\theta\in[1/2, 1] $ by
\begin{eqnarray} 
\label{repform3}
c(x,t_{n+1})&\approx &{\mathcal E}_{\Delta t}[c(.,t_{n})](x)    \\
&&+(1-\theta){\Delta t}  {\mathcal E}_{\Delta t}[f\circ c](x)  + \theta{\Delta t}f(c(x,t_{n+1})).
\nonumber
\end{eqnarray}


The goal of this work is to extend this approach to the case of a semi-Lagrangian treatment of
both advection and diffusion terms. In order to do this, we introduce a spatial discretization
mesh $\mathcal{G}_{\Delta x}=\{x_i, \; x_i \in  \Omega\},$ where $\Delta x$ denotes a measure of the mesh resolution. 
%
Notice that the only necessary restriction on the nature
of the mesh is  that  it should be possible to define a polynomial interpolation operator $I_p[c](\cdot)$ of degree $p$, constructed on the values of  a grid function $c$ defined on $\mathcal{G}_{\Delta x}$.
Assuming in the following that such operator is defined on structured and unstructured
meshes  (for the definition of such interpolation operator see for instance \cite{quarteroni:1994}),  examples of applications using both 
meshes will be considered.\\
We then denote the discrete approximations
of the solution of \eqref{eq:adr} at the space-time mesh nodes by $c_i^n $
and by
\begin{equation}\label{z_i}
z_i^{n+1}=X_{\Delta t}(x_i,t_{n+1};t_{n})
\end{equation}
some  numerical approximations of the streamlines starting at $x_i$  and defined by the velocity field $ u $ over the interval
 $[t_{n}, t_{n+1}].$ In particular, in \cite{bonaventura:2014},\cite{bonaventura:2018}
 explicit Euler or Heun methods were employed to compute these approximations,
 coupled to a substepping approach along the lines of \cite{casulli:1990},\cite{rosatti:2005}.
 More specifically, for the Euler method, 
 given a positive integer $m$,  a time substep was defined as
  $\Delta \tau=\Delta t/m$ and, for $q=0,\dots,m-1$,
\begin{eqnarray}
\label{substep_1st}
y^{(q+1)}&=&y^{(q)}-\Delta \tau u(y^{(q)},t_{n}) , \nonumber \\
y^{(0)}&=&x_i
\end{eqnarray}
was computed, so that $ z_i^{n+1}=y^{m}.$
 Following the approach outlined in \cite{bonaventura:2014} and applied in \cite{bonaventura:2018},
a first order in time approximation of the solution of \eqref{eq:adr} can then be defined as
\begin{eqnarray} 
\label{first_order}
c_i^{n+1}  &= &    \frac{1}{4} \sum_{k=1}^4 I_p[c^n] \left( z_i^{n+1}+ \delta_k \right)  \\
&&+(1-\theta)\Delta t \frac{1}{4} \sum_{k=1}^4 f(I_p[ c^n]) \left(  z_i^{n+1}+ \delta_k \right) 
 + \theta \Delta tf(c_i^{n+1}) ,
\nonumber
\end{eqnarray}
where $c^n=(c^n)_i$. In \eqref{first_order}, $\delta_k$ is defined, for $k=1,\ldots,4$, as
$$
\delta_k = \pm \delta  e_j
$$
($e_j$ $(j=1,2)$ denoting the canonical basis in $\R^2$), for all combination of both the sign and the index $j=1,2$; moreover,
$\delta=\sqrt{4\nu\Delta t}$.

Notice that, for simplicity, we neglect in \eqref{first_order} the treatment of boundary conditions. Possible approaches
to handle nontrivial boundary conditions will be discussed in Section \ref{bc}.

The method \eqref{first_order} will be denoted in what follows by SL1.
This method inherits the same stability and convergence properties of the parent methods,
as it will be discussed in Section \ref{convergence}. Notice that this approach can be extended
to spatially varying diffusion coefficients and that, while only first order in time, its effective accuracy can
be substantially superior to that of more standard techniques, if higher degree interpolation operators
are used, as shown in \cite{bonaventura:2014}.

 In order to derive a method of second order in time,  we follow the main steps of
   \cite{ferretti:2010},\cite{milstein:2013}.
 The streamlines  are interpreted  as  generalized characteristics, defined as the solutions of
 the stochastic differential equation:
$$\mathrm{d} X=-  u(X(s),s) \dd s+\sigma \dd W(s),\quad  X(t_{n+1})= x_i$$
for $s\in [0,\Delta t]$,  $\sigma=\sqrt{2 \nu}, $ where $ W(s) $ denotes a standard 2-dimensional Wiener process.
Applying the implicit weak method of order 2 defined in \cite{kloeden:1992} for the approximation of the streamlines, 
 we denote by 
 \begin{equation}
 \label{disp2}
 z^{n+1}_{k,i}= X^k_{\Delta t}(x_i,t_{n+1};t_n)
  \ \ \ \ (k=1,\dots,9)
  \end{equation}
the solutions of the nonlinear equations
\begin{equation}
\label{cnstream}
z^{n+1}_{k,i}= x_i-\frac{\Delta t}{2}\left(u(x_i,t_{n+1})+u(z^{n+1}_{k,i},t_{n}) \right)+\sqrt{3\Delta t}\sigma  e_k.
\end{equation}
Here, $e_k$ denote the vectors:
$$
e_1=\left(\begin{array}{c}
0\\
0
\end{array}\right),\quad
e_2=\left(\begin{array}{c}
0\\
1
\end{array}\right),\quad
e_3=\left(\begin{array}{c}
 0\\
-1
\end{array}\right),$$
$$
e_4=\left(\begin{array}{c}
1\\
0
\end{array}\right),\quad
e_5=\left(\begin{array}{c}
-1\\
 0
\end{array}\right),\quad
e_6=\left(\begin{array}{c}
1\\
1
\end{array}\right),$$
$$
e_7=\left(\begin{array}{c}
 1\\
-1
\end{array}\right),\quad
e_8=\left(\begin{array}{c}
-1\\
 1
\end{array}\right),\quad
e_9=\left(\begin{array}{c}
-1\\
-1
\end{array}\right).$$
These vectors represent a realization of a vector random variable whose distribution is given by the probabilities
$$\alpha_1=4/9,\quad \alpha_2=\alpha_3=\alpha_4=\alpha_5=1/9,
\quad\alpha_6=\alpha_7=\alpha_8=\alpha_9=1/36.$$
It is to be remarked that   method \eqref{cnstream} can be rewritten in terms of the displacements
$\delta^{n+1}_{k,i}=z^{n+1}_{k,i}-x_i $ as
\begin{equation}
\label{cnstream_disp}
 \delta^{n+1}_{k,i}=-\frac{\Delta t}{2}\left(u(x_i,t_{n+1})
+u(x_i+\delta_{k,i},t_{n}) \right)+\sqrt{3\Delta t}\sigma e_k,
\end{equation}
thus yielding an implicit method that is a natural extension to stochastic differential equations
of that introduced in \cite{robert:1982} and commonly used in meteorological applications for the computation
of streamlines in SL methods.
A second order in time SL (SL2) scheme can then be defined by a Crank-Nicolson approach as
\begin{eqnarray}
 \label{second_order_CN}
c_i^{n+1} &=&      \sum_{k=1}^9 \alpha_k \left( I_p[c^n] (z^{n+1}_{k,i})+\frac{\Delta t}{2}  f ( I_p[c^n] ( z^{n+1}_{k,i}))\right)
\nonumber \\
 &+&  \frac{\Delta t}2f(c_i^{n+1}).
\end{eqnarray}
Notice that, with respect to the simpler first order in time  variant \eqref{first_order}, 
nine interpolations at the foot of the streamlines must be computed,  which clearly makes
this approach substantially more expensive than the simpler, first order in time variant.
In applications to systems of the form \eqref{eq:adrsys}, the computational cost of scheme  \eqref{second_order_CN}
can be marginally reduced by setting
$$ 
\tilde c_i^{n}=  \sum_{k=1}^9 \alpha_k I_p[c^n] (z^{n+1}_{k,i})
$$
and defining
\begin{equation}
 \label{second_order_CN_nl}
c_i^{n+1} = \tilde c_i^{n}  +\frac{\Delta t}{2}  f(\tilde c_i^{n})+  \frac{\Delta t}2f(c_i^{n+1}),
\end{equation}
so as to reduce the number of the evaluations of a possibly costly nonlinear term.
Furthermore, when the coupling of the diffusion and  advection term is weak, it should be possible to decouple again the 
approximation of a  single deterministic streamline from that of the diffusive displacements,
which could be added at the end of each approximate streamline without increasing too much the error.
The deterministic streamline can then be computed by a substepping procedure, applied  either
to explicit methods as in \eqref{substep_1st}  or to implicit methods as in \eqref{cnstream}.
  
For example, a decoupled substepping variant of \eqref{cnstream} might be obtained by computing, for $q=0,\dots,m-1$,
\begin{eqnarray}
\label{substep_2nd}
 y^{(q+1)}&=& y^{(q)}\nonumber \\
&&-\frac{\Delta \tau}{2}\left[u(y^{(q)},t_{n+1}-q\Delta \tau)
+ u(y^{(q+1)},t_{n+1}-(q+1)\Delta \tau) \right],
\nonumber \\
y^0&=&x_i,
\end{eqnarray}
and setting $z_k =y^m+\sqrt{3\Delta t}\sigma e_k.$
 We will denote this  decoupled variant with substepping by SL2s.
 
Notice that, in realistic problems, a bottleneck of the scheme \eqref{second_order_CN} is the fact that the Crank--Nicolson method, while A-stable, is not L-stable, see e.g. \cite{lambert:1991}. Therefore, no damping is introduced by the method for very
large values of the time step and spurious oscillations may arise, see also the discussion in 
\cite{bonaventura:2017}. In order to reduce the computational cost and to address the L-stability issue, 
different variants of the  \eqref{second_order_CN} scheme could also be introduced and compared, along the lines
proposed in \cite{tumolo:2015} for the pure advection case. However, this development goes beyond the scope of this paper and will not be pursued here.

 Finally, even though achieving full second order consistency is quite complicated in the variable diffusion coefficient case,
 the previously introduced schemes can be nonetheless extended at least in the simpler configurations
 as suggested in \cite{bonaventura:2014} for the first order case, even though full second order accuracy is not
 guaranteed any more.
%
 
 \section{Convergence  analysis} \label{convergence}

We present in this section a convergence analysis for the scheme \eqref{second_order_CN}.
For simplicity, we assume a one-dimensional problem defined on $\mathbb{R}\times[0,T]$, with a time-independent drift term $u$:
\begin{equation}\label{eq:adr1d}
\begin{cases}
c_t+ u(x) c_x - \nu c_{xx} = f(c) & (x,t)\in \mathbb{R} \times \left( 0,T \right], \\
c (x,0 )=c_0 ( x) &  x\in\mathbb{R}.
\end{cases}
\end{equation}
The multidimensional case, as well as the time dependence of $u$, require only small technical adaptations. On the other hand, the convergence analysis on bounded domains is still an open problem for high-order SL schemes, therefore we will not address this problem here.\\
First, for $i\in \Z$ and $n=0,\dots,N-1$, we rewrite the scheme \eqref{second_order_CN} with the shorthand notation
\begin{equation}\label{eq:SL2compact}
c_i^{n+1} =S_{\Delta t, \Delta x} \left( c^{n+1}, c^n, x_i \right),
\end{equation}
where $x_i=i\Delta x$, and 
\begin{eqnarray*}
S_{\Delta t, \Delta x} \left( c^{n+1}, c^n, x_i \right)&=& \sum_k \alpha_k \left[I_p [ c^n ](z_k (x_i ) ) + \frac{\Delta t}{2} f(I_p [c^n]( z_k (x_i )) \right] \nonumber \\
&&+\frac{\Delta t}{2} f \left( c_i^{n+1}\right).
\end{eqnarray*}
In one space dimension, the three discrete characteristics are defined by the equations
\begin{eqnarray*}
z_+(x) &=& x - \frac{\Delta t}{2}[u ( x ) + u ( z_+(x)) ]+ \sqrt{6 \Delta t \nu}, \\
z_{-}(x) &=& x - \frac{\Delta t}{2} [ u( x ) + u ( z_-(x) )] - \sqrt{6 \Delta t \nu}, \\
z_0(x) &=& x - \frac{\Delta t}{2}[u( x ) + u ( z_0(x) )],
\end{eqnarray*}
with corresponding weights $\alpha_+=\alpha_-=1/6$ and $\alpha_0=2/3$.
In what follows, we will use the symbol $K$ to denote various positive constants, which do not depend on $\dt,x,t$. We will also assume that:
\begin{itemize}
\item[\bf(H0)] there exists a unique classical solution of  \eqref{eq:adr1d};
\item[\bf(H1)] $f(x)\in C^4(\mathbb{R})$ with $|f^{(p)}(x)|\leq K $ for $p\leq 4;$ 
\item[\bf(H2)] $u(x)\in C^2(\mathbb{R})$ with $|u^{(p)}(x)| \leq K$  for $p\leq 2;$ 
\item[\bf(H3)] for any $v(x)\in C^{p+1}(\mathbb{R})$ with  bounded derivatives, $I_p[v]$ is a piecewise polynomial interpolation operator such that for any $x\in \mathbb{R}$
$$| I_p [v](x)-v(x)|\leq K \Delta x^p.$$ 
\end{itemize}

\subsection{Consistency}
First, we derive a consistency result via a  Taylor expansion. The same kind of result can be obtained by probabilistic arguments, see \cite{milstein:2000}.
\begin{Prop}\label{Prop:consistency} Assume {\bf(H1)}--{\bf(H3)}, and let $c(x,t)$ be a smooth solution of \eqref{eq:adr1d}. Then, for each $(i,n)\in \Z\times\{0,\dots,N-1\}$ the consistency error  of the scheme \eqref{second_order_CN}, defined as
$$\mathcal{T}_{ \Delta t,\Delta x} (x_i,t_{n})= 
\frac{1}{\dt}\left(c(x_i,t_{n+1})-
S_{\Delta t, \Delta x} \left( c(t_{n}), c(t_{n+1}), x_i \right)\right)$$
where  $c(t_{n})=(c(x_i,t_n))_i$,  is such that
$$\mathcal{T}_{ \Delta t,\Delta x} (x,t) =  \mathcal{O} \left ( \Delta t^2+\frac{\Delta x^p}{\dt} \right).$$
\end{Prop}
\begin{proof}
In what follows, we will omit the argument of  functions  computed at $(x,t)$. Consider a smooth solution $c$ of \eqref{eq:adr1d}.  Since assumption {\bf(H1)} holds, by differentiating in time and space \eqref{eq:adr1d} we get that $c$ is also solution of
\begin{equation}\label{eq:tt}
c_{tt} + u \left( x \right) c_{xt} - \nu c_{xxt} = f'(c)c_t,
\end{equation}
\begin{equation}\label{eq:2}
c_{tx} + u'(x) c_x + u \left( x \right) c_{xx} - \nu c_{xxx} = f'(c)c_x,
\end{equation}
and hence, by differentiating again in space \eqref{eq:2}, of 
\begin{eqnarray}\label{eq:3}
&&c_{txx} + u''(x) c_x + u'(x) c_{xx} + u'(x) c_{xx} + u \left( x \right) c_{xxx} - \nu c_{xxxx} \nonumber\\
 &&= f''(c)(c_x)^2+f'(c)c_{xx}.
\end{eqnarray}
Using  \eqref{eq:2} and \eqref{eq:3} in \eqref{eq:tt}, we get :
\begin{eqnarray}\label{eq:ctt}
 c_{tt} & = &u u' c_c + u^2 c_{xx} - u \nu c_{xxx} - \nu u'' c_x -2 \nu u' c_{xx} - u\nu c_{xxx} + \nu^2 c_{xxxx} \nonumber\\ 
& = &\left( u u' - \nu u' \right) c_x + \left( u^2 - 2 \nu u' \right) c_{xx} - 2 u \nu c_{xxx} + \nu^2 c_{xxxx} \nonumber\\
& &+f'(c)(c_t-uc_x+\nu c_{xx})+\nu f''(c)(c_x)^2.
\end{eqnarray}
Define now $U_{\pm}(x) =  u \left( x \right) + u \left( z_{\pm}(x)\right)$.
By a Taylor expansion of $c( z_{\pm}(x), t )$ in space around $(x,t)$, we obtain
\begin{eqnarray}\label{eq:sviluppoc}
c ( z_{\pm}(x), t ) &=&c + \left(\pm \sqrt{6 \Delta t \nu} - \dtm U_\pm \right) c_x + \frac{1}{2} \left( \pm\sqrt{6 \Delta t \nu} - \dtm  U_\pm \right)^2 c_{xx}  \nonumber \\ 
&&+\frac{1}{6} \left( \pm\sqrt{6 \Delta t \nu} - \dtm  U_\pm \right)^3 c_{xxx} \nonumber \\
&&+ \frac{1}{24}  \left(\pm \sqrt{6 \Delta t \nu} - \dtm  U_\pm  \right)^4 c_{xxxx} \nonumber \\
&&+\frac{1}{120}  \left(\pm \sqrt{6 \Delta t \nu} - \dtm  U_\pm  \right)^5 c_{xxxxx} + \mathcal{O}( \Delta t^3 ) 
\end{eqnarray}
and, defining $U_{0}(x) =  u \left( x \right) + u \left( z_{0}(x)\right)$,
\begin{equation}\label{eq:sviluppoc0}
c \left( z_0(x), t \right) = c- \dtm U_0 c_x + \frac{1}{2} \left(- \dtm  U_0  \right)^2 c_{xx} + \mathcal{O} \left( \Delta t^3 \right). 
\end{equation}
Using \eqref{eq:ctt},\eqref{eq:sviluppoc},\eqref{eq:sviluppoc0} and  the Taylor expansion  
$$c ( x, t + \Delta t ) = c + \Delta t c_t + \frac{\Delta t^2}{2} c_{tt} + \mathcal{O} \left( \Delta t^3 \right),$$
 we obtain
%
\begin{eqnarray}\label{err_senzaf}
c \left(x, t + \Delta t \right) & - & \sum_k \alpha_k c (z_k (x),t) = \dt(c_t + u \left( x \right) c_x - \nu c_{xx}) 
\nonumber \\
&+&\frac{\dt^2}{2}(f'(c)(c_t-uc_x+\nu c_{xx})+\nu f''(c)(c_x)^2) \nonumber \\
&+& \mathcal{O} \left ( \Delta t^3 \right).
\end{eqnarray}
Consider now the nonlinear reaction term. By assumption {\bf(H3)}, we have that
\begin{eqnarray}\label{f_sviluppo}
f(c( z,t))=f(c)+f'(c)(c( z,t) -c)+\frac{1}{2}f''(c)(c ( z,t) -c)^2 \nonumber \\
+\frac{1}{6}f'''(c)(c ( z,t) -c))^3+\mathcal{O} \left( (c( z,t) -c)^4 \right) 
\end{eqnarray}
Using \eqref{err_senzaf} in \eqref{f_sviluppo}, and taking into account that $c(z_\pm,t)=c \pm\sqrt{6\dt \nu} c_x+ \mathcal{O}(\dt)$ and  $c ( z_0(x),t) =c+ \mathcal{O}(\dt)$, we obtain
\begin{eqnarray}\label{f_cons}
\sum_k(\alpha_k  f( c( z_k (x ),t))&=& f(c)
+f'(c)\left(c(x,t+\dt)-\dt(c_t+uc_x-\nu c_{xx})-c\right)\nonumber \\
 &+& f''(c)(\dt \nu c^2_x)+ \mathcal{O}(\dt^2). 
\end{eqnarray}
By \eqref{f_cons} and \eqref{err_senzaf}, we get  the consistency error  for the semi-discretization,
\begin{eqnarray}\label{err_cons}
c (x, t + \Delta t) &-&\sum_k\alpha_k\left( c (z_k (x) ,t)+ \frac{\Delta t}{2}  f( c( z_k (x ),t))\right)\nonumber\\
&-&\frac{\Delta t}{2} f( c(x_,t+\dt ))\nonumber \\
&=& \dt(c_t + u( x) c_x - \nu c_{xx}-f(c))+ \mathcal{O} \left ( \Delta t^3 \right).
\end{eqnarray}
Introducing the interpolation error and using assumptions {\bf(H2)} and {\bf(H1)}, we finally prove the consistency error for the fully discrete scheme.
\end{proof}
\subsection{Stability}
To prove stability, it is convenient to recast \eqref{eq:SL2compact} in the matrix form
\begin{equation*}
c^{n+1} - \frac{\Delta t}{2} f \left( c^{n+1}\right) = \sum_k \alpha_k \left[ B_k c^n + \frac{\Delta t}{2} f \left( B_k c^n \right) \right],
\end{equation*}
where, for a vector $c$, $f(c)$ denotes the vector obtained by applying $f$ elementwise and the matrices $B_k$ (which represent the operation of interpolating $c^n$ at the points $z_k(x_i)$) have elements $b_{k,ij}$ defined by
\begin{equation}\label{def:D}
b_{k,ij}= \psi_j(z_k(x_i)),
\end{equation}
for a suitable basis of cardinal functions $\{\psi_j\}$.
The following proposition implies stability for the linear part of the scheme with respect to the 2-norm. For a formal definition of the basis functions $\psi_j$ we refer the reader to \cite{ferretti:2013a}, \cite{FM19}.
\begin{Prop}\label{Prop:stability} Assume {\bf(H2)}, and let the matrix $B$ have elements defined by \eqref{def:D}, with $(\psi_j)$ basis functions for odd degree symmetric Lagrange or splines interpolation. Then,  for each $k$, there exists a constant $K_B>0$ independent on $\Delta x, \dt$ such that
\begin{equation}\label{eq:03}
\|B_k\| \le 1+K_B\Delta t.
\end{equation}
\end{Prop}
\begin{proof}
If we adopt a piecewise linear interpolation, the scheme is monotone and satisfies \eqref{eq:03} both in the $\infty$-norm (with $K_B=0$) and in the 2-norm, see for example \cite{ferretti:2013a}. We will therefore focus on the case of high-order interpolations, for which the norm in \eqref{eq:03} should be understood as 2-norm. 
Following \cite{ferretti:2013a},  \cite{FM19}, we sketch the arguments to prove \eqref{eq:03} for the cases of symmetric Lagrange and splines interpolation, which can be interpreted as Lagrange--Galerkin schemes with area-weighting. 
First, we make explicit the dependence of the points $z_k$ on $x$ and $\Delta t$. We recall that $z_k(x)=x+\delta_k(x)$, with $\delta_k$ solving the equation:
\[
\delta_k(x) = -\frac{\Delta t}{2}\left(u(x) + u(x+\delta_k(x))\right) + \sqrt{3\Delta t}\sigma e_k.
\]
Expanding the term $u(x+\delta_k(x))$, we obtain therefore
\[
\delta_k(x) = -\frac{\Delta t}{2}\left(u(x) + u(x)+\delta_ku'(x)) + \mathcal{O}\left(\delta_k^2\right)\right) + \sqrt{3\Delta t}\sigma e_k,
\]
and hence,
\[
\delta_k(x)\left(1 + \frac{\Delta t}{2}u'(x)\right) = -\Delta t \> u(x) + \sqrt{3\Delta t}\sigma e_k + \mathcal{O}\left(\Delta t\delta_k^2\right).
\]
Dividing now by $1 + \Delta t\>u'/2$, and using the fact that $\delta_k=\mathcal{O}(\sqrt{\Delta t})$, we get, for $\Delta t\to 0$,
\begin{equation}\label{eq:z_k}
z_k(x) = x - \Delta t \> u(x) + \sqrt{3\Delta t}\sigma e_k + \mathcal{O}\left(\Delta t^2\right).
\end{equation}
Due to the term $\sqrt{3\Delta t}\sigma e_k$,  the form of  \eqref{eq:z_k}  does not coincide with that
used in \cite{ferretti:2013a} for the points $z_k(x)$. However, when considering differences $z_k(x_1)-z_k(x_2)$, this additional term is cancelled, so that
\[
z_k(x_1)-z_k(x_2) = (x_1-x_2) - \Delta t (u(x_1)-u(x_2)) + \mathcal{O}\left(\Delta t^2\right).
\]
As a consequence, the form \eqref{eq:z_k} still satisfies the relevant properties used in the proof of \eqref{eq:03}. In particular, it satisfies the condition \cite[Lemma 3]{ferretti:2013a}
\[
|z_k(x_1)-z_k(x_2)| \ge \frac 1 2 |x_1-x_2|,
\]
for $\Delta t$ small enough, as well as the condition \cite[Theorem 4]{ferretti:2013a}
\[
|z_k(x_1)-(x_1-x_2+z_k(x_2))| \le K_X |x_1-x_2|\Delta t,
\]
for a suitable positive constant $K_X$. Then, a careful replica of the arguments used in \cite{ferretti:2013a} provides the estimate \eqref{eq:03}.
\end{proof}
\subsection{Convergence}
We now present a convergence result in the discrete 2-norm.
\begin{Theorem} 
 Assume {\bf (H0)}--{\bf (H3)}.  Let $c(x,t)$ be the classical solution of \eqref{eq:adr1d}, and $c^n$ be the solution of \eqref{eq:SL2compact}. Then, for any $n$ such that $t_{n}\in[0,T]$ and for $(\dt, \Delta x)\to 0$,
$$
\| c(t_n)-c^n \|_2 \leq K_T\left(\Delta t^2+\frac{\Delta x^p}{\Delta t}\right),
$$
where $K_T$ is positive constant depending on the final time $T$.
\end{Theorem}
\begin{proof}
While a mere convergence proof could be carried out with weaker regularity assumptions, we will focus here on the error estimate above, which require the regularity assumed in {\bf (H0)}--{\bf (H3)}.
Define the vectors $\gamma^n$ and $\nu^n$, so that $\gamma_i^n=c(x_i,t_{n})$, and $\epsilon^{n}  =\gamma^n - c^{n}$. Then,  by Proposition \ref{Prop:consistency}, we get
\begin{eqnarray}\label{eq:04}
\gamma^{n+1} - \frac{\Delta t}{2} f \left(\gamma^{n+1}\right) & = & \sum_k \alpha_k \left[ B_k^n \gamma^n + \frac{\Delta t}{2} f \left( B_k^n \gamma^n \right) \right] \nonumber \\
&&+ \mathcal{O}(\Delta t^3+\Delta x^p).
\end{eqnarray}
Subtracting \eqref{eq:SL2compact} from \eqref{eq:04}, using the Lipschitz continuity of $f$ and the triangle inequality in the form of a difference, we obtain from the left-hand side:
\begin{equation*}
\left\|\gamma^{n+1} - \frac{\Delta t}{2} f \left(\gamma^{n+1}\right) - c^{n+1} + \frac{\Delta t}{2} f \left(c^{n+1}\right)\right\|_2 \ge \left(1-\frac{L_f\Delta t}{2}\right)\left\|\epsilon^{n+1}\right\|_2.
\end{equation*}
Taking into account that $\sum_k\alpha_k=1$, along with the bound \eqref{eq:03}, we also have from the right-hand side:
\begin{eqnarray*}
&&\left\|\gamma^{n+1} - \frac{\Delta t}{2} f \left(\gamma^{n+1}\right) - c^{n+1} + \frac{\Delta t}{2} f \left(c^{n+1}\right)\right\|_2 \\
 &&\le\left(1+\frac{L_f\Delta t}{2}\right)(1+K_B\Delta t)\left\|\epsilon^n\right\|_2 + \mathcal{O}(\Delta t^3+\Delta x^p).
\end{eqnarray*}
Therefore, it turns out that
\begin{eqnarray}\label{eq:05}
\left(1-\frac{L_f\Delta t}{2}\right)\left\|\epsilon^{n+1}\right\|_2 & \le & \left(1+\frac{L_f\Delta t}{2}\right)(1+K_B\Delta t)\left\|\epsilon^n\right\|_2 \nonumber \\
&& + \mathcal{O}(\Delta t^3+\Delta x^p).
\end{eqnarray}
Now, for $\Delta t$ small enough to have $1-L_f\Delta t/2>\underline C > 0$, we have that there exists a constant $K_T>0$ such that
\[
\frac{\displaystyle 1+\frac{L_f\Delta t}{2}}{\displaystyle 1-\frac{L_f\Delta t}{2}}(1+K_B\Delta t) \le 1+K_T\Delta t,
\]
and hence, using this bound in \eqref{eq:05},
\begin{equation}\label{eq:06}
\left\|\epsilon^{n+1}\right\|_2 \le (1+K_T\Delta t)\left\|\epsilon^n\right\|_2 + \mathcal{O}(\Delta t^3+\Delta x^p),
\end{equation}
which, by standard arguments, implies that, for any $n$ such that $t_{n}\in[0,T]$,
\[
\left\|\epsilon^n\right\|_2 \le K_T\left(\Delta t^2+\frac{\Delta x^p}{\Delta t}\right).
\]
\end{proof}
\section{Boundary conditions}
 \label{bc}
 
The treatment of Dirichlet boundary conditions for this class of semi-Lagrangian methods has been considered  in 
\cite{milstein:2001}, where two methods are proposed. One approach has first order of consistency, but it does not seem 
possible to generalize it to multiple dimensions. The second approach has order of consistency $1/2.$ More recently, in \cite{bonaventura:2018}, 
an easier treatment has been proposed for the scheme SL1 with time-independent Dirichlet boundary condition, again with order of consistency $1/2.$ This approach has been extended in \cite{bonaventura:2020} to unstructured meshes.

We propose here a new approach to obtain second order consistency for the scheme SL2 with Dirichlet boundary conditions. This technique is based on extrapolation, much in the spirit of the so-called {\em ghost-point} techniques,
see e.g. \cite{leveque:2002}.
 

In addition to the standard mesh $\mathcal{G}_{\Delta x}=\{x_i, \; x_i \in \overline \Omega\},$  where  the numerical solution is computed, we consider a second mesh $\mathcal{G}_h=\{v_i, \; v_i \in \overline \Omega\}$ formed by a single layer of elements having their external side along the boundary
of   $  \Omega.$ This second mesh is constructed with a size parameter $h\sim \sqrt{\Delta t}$, and the degrees of freedom are chosen in order to allow a second-order interpolation.
In Fig.\ref{fig:boundarymesh} we show, as an example,  a square domain $\Omega=(-1,1)\times(-1,1)$, for which the standard mesh $\mathcal{G}_{\Delta x}$  is formed by the blue triangular elements and the mesh $\mathcal{G}_h$  is formed by the black rectangular elements. Note that the latter overlap at the corners. The asterisks in red are the nodes of $\mathcal{G}_h$, according to the standard $\mathbb{Q}_2$ element. The values of the numerical solution on the nodes $v_i$ are obtained by interpolation at internal nodes, and by the Dirichlet boundary condition if the nodes lie on the boundary $\partial \Omega$.


\begin{figure}
	\centering
	\includegraphics[width=1\textwidth]{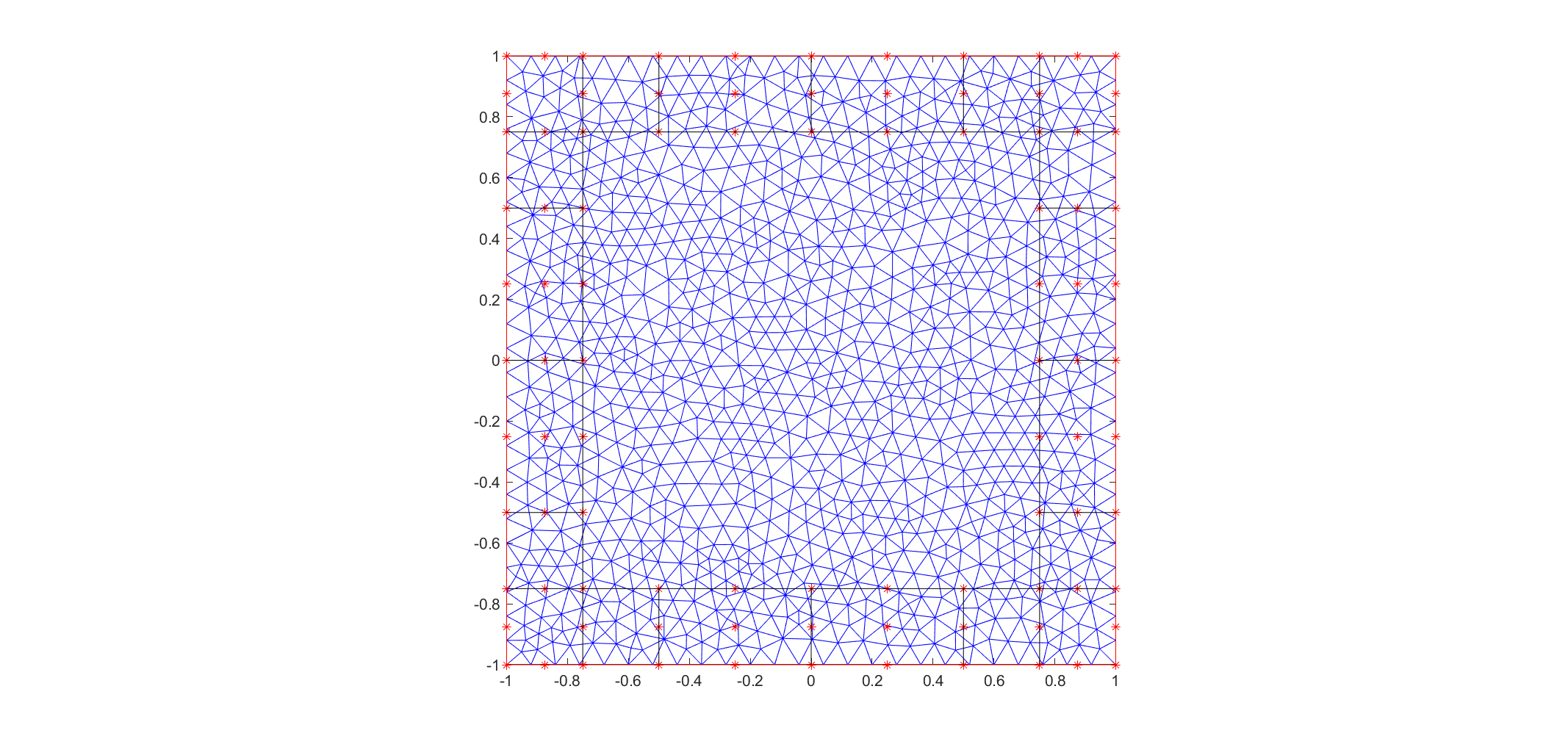}
	\caption{Unstructured computational mesh (blue triangular elements) together with boundary mesh $\mathcal{G}_{h}$ (black  rectangular elements and red asterisks nodes) }
	\label{fig:boundarymesh}
\end{figure}
Let us call $\mathcal{T}_{\Delta x}$  a given triangulation, with ${\mathcal{G}}_{\Delta x}$   the set of the  vertexes of the elements $K\in\mathcal{T}_{\Delta x}$. Let us define the polygonal domain  $\Omega_{\Delta x}:=\cup _{K \in \tau_{\Delta x}}K\subset \Omega $.
If, for some $i$ and $k$, $z^{n+1}_{k,i}\notin \overline \Omega_{\Delta x}$, then its projection $P(z^{n+1}_{k,i})$ 
onto $\overline \Omega_{\Delta x}$ is computed, defined as the point in $\overline \Omega _{\Delta x}$ at minimum distance
from $z^{n+1}_{k,i}.$  The value of the numerical solution $c^n(z^{n+1}_{k,i})$ is then approximated by a quadratic extrapolation operator $\Psi_2$. This operator is constructed  via the $\mathbb{Q}_2$ associated to the element of $\mathcal{G}_h$ to which the projection $P(z^{n+1}_{k,i})$ belongs:
 $$
 c^n(z^{n+1}_{k,i})\simeq \Psi_2[\hat c^n](z^{n+1}_{k,i})),
 $$
 where $\hat c^n$  corresponds to
 $$
 \hat c^n(v_i)=\begin{cases}I_2[c^n](v_i)&  {\textrm{if}}  \;v_i \in \Omega,\\
 b(v^i,t_{n}) & {\textrm{if}} \;v_i \in \partial \Omega.\\
\end{cases}
 $$
 The method can be extended to more general domains, by considering triangular elements for $\mathcal{G}_h$.
 This technique will not be rigorously analysed here, but we will provide a numerical validation in Section \ref{test:boundarycondition}.
 
  \section{Numerical results}
 \label{tests}
 
 A number of numerical experiments have been carried out, in order to assess the accuracy of the proposed methods on both structured and unstructured meshes. 
In the unstructured case, we have constructed a triangular mesh by the Matlab2019 function {\tt{generateMesh}}, with a maximum mesh edge of $\Delta x$, and used a ${\P}_2$ space reconstruction. In the structured Cartesian case, the bicubic polynomial interpolation implemented in the Matlab2019 command {\tt{interp2}}, has been used.
Since the goal is to evaluate the accuracy of time discretization, both choices avoid to hide the time discretization error with the error introduced by a lower order space reconstruction.

\subsection{Pure diffusion}
\label{purediff}

In a first, basic test, we consider equation \eqref{eq:adr} in the pure diffusion case, i.e., with zero advection and reaction terms, on the square domain $\Omega=(-2,2)\times(-2,2)$, with $T=1$ and $\nu=0.05$.
 Based on the test case proposed in \cite{pudykiewicz:1984},  
 we assume a Gaussian  initial datum centered in $(0,0)$, with $\sigma= 0.1$, so that the exact solution
 in an infinite plane would be
 $$
 c(x,y,t)=\frac{1}{1+2\nu t/\sigma^2}
 \exp\left \{- \frac{x^2+y^2}{2(\sigma^2+2\nu t)}\right \}.
 $$
 For this test case, we only consider structured meshes with constant steps $\Delta x=\Delta y = 4/N $ in both directions.
 Following \cite{bonaventura:2014}, we consider different time step values $\Delta t$, which correspond  to
 different values of the parabolic stability parameter $\mu= \Delta t\nu/\Delta x^2$.
We compare method SL1 \eqref{first_order} and method SL2 \eqref{second_order_CN}, and collect the results in Table \ref{table2}. It can be observed that the expected convergence rates are recovered. Furthermore, it is apparent that scheme SL2
yields a substantial accuracy improvement, without an excessive increase in computational cost.
Indeed, the SL2 runs require between 30\% and 60\% more CPU time, depending on the resolution, while leading
to corresponding error reductions between 140\% and 730\%. As a comparison, a standard second order discretization
in space coupled to an explicit second order method in time yields at the finest resolution an error 5 times larger than that
of method SL2 at approximately the same computational cost.

 \begin{table}
  \centering
\begin{tabular}{||c|c|c|c|c|c||} 
\hline
\multicolumn{2}{||c|}{Resolution} & \multicolumn{2}{c}{Relative error}  & \multicolumn{2}{|c||}{Convergence rates} \\
\hline
$\Delta x$ &   $\mu $ & $l_2$ &  $l_{\infty}$ & $p_2$ &  $p_{\infty}$   \\
\hline 
$0.08$ &    $0.84$ & $3.34\cdot 10^{-2}$& $5.10\cdot 10^{-2}$  & - & -  \\
\hline
$0.04$ & $1.6$ & $1.33\cdot 10^{-2}$ & $2.05\cdot 10^{-2}$  &  1.33 &  1.00  \\
\hline
$0.02$ &    $3.2$ & $6.57\cdot 10^{-3}$ & $1.03\cdot 10^{-2}$   & 1.02 &  0.99 \\    
\hline
\end{tabular}

\bigskip

\begin{tabular}{||c|c|c|c|c|c||} 
\hline
\multicolumn{2}{||c|}{Resolution} & \multicolumn{2}{c}{Relative error}  & \multicolumn{2}{|c||}{Convergence rates} \\
\hline
$\Delta x$ &   $\mu $ & $l_2$ &  $l_{\infty}$ & $p_2$ &  $p_{\infty}$   \\
\hline 
$0.08$ &    $0.84$ & $2.66\cdot 10^{-3}$ & $4.76\cdot 10^{-3}$  & - & -  \\
\hline
$0.04$ & $1.6$ & $4.89\cdot 10^{-4}$ & $8.24\cdot 10^{-4}$  & 2.44 &   2.53 \\
\hline
$0.02$ &    $3.2$ & $8.89\cdot 10^{-5}$ & $1.48\cdot 10^{-4}$  & 2.46  & 2.48 \\    
\hline
\end{tabular}
\caption{Errors for the pure diffusion test, first order method SL1 (upper) and second order method SL2 (lower) on a structured mesh.}
\label{table2} 
 \end{table}
 \subsection{Solid  body rotation}
 \label{solidbody}
 Next, we consider the advection--diffusion equation \eqref{eq:adr}  with 
 coefficients $u=(-\omega y,\omega x), $  $\omega=2\pi, $  $\nu=0.05 $  and $f=0 $
 on the square domain $\Omega=(-2,2)\times(-2,2)$ and $T=1$.
 Following \cite{pudykiewicz:1984}, we assume a Gaussian  initial datum centered at $(x_0,y_0)=(1,0)$ with $\sigma= 0.05$, so that the exact solution in an infinite plane would be
 \begin{equation}
 c(x,y,t)=\frac{1}{1+2\nu t/\sigma^2}\exp\left \{-\frac{(x-x(t))^2+(y-y(t))^2}{2(\sigma^2+2\nu t)}\right \},
 \end{equation}
where $x(t)=x_0\cos{\omega t}-y_0\sin{\omega t}, $ $y(t)=x_0\sin{\omega t}-y_0\cos{\omega t}.$ 
We first consider  structured meshes with constant steps $\Delta x=\Delta y = 4/N $ in both coordinate directions.
We consider again values of $\Delta t$ corresponding to
 different values of the parabolic stability parameter $\mu$, as well as of the Courant number $\lambda=\Delta t\max{|{u}|}/\Delta x$.  
 
In the structured case, we compare method SL1 \eqref{first_order} with Euler substepping as in \eqref{substep_1st}, the decoupled variant SL2s of method \eqref{second_order_CN} with Heun substepping, and method SL2 \eqref{second_order_CN} with the
fully coupling \eqref{cnstream_disp}.   
The results are reported in Table \ref{tab:test2SL2}, in which convergence rates are computed with respect to the
values in the first row. Furthermore, the convergence rate estimation for the values
in the last row takes into account that the time step has been reduced by a factor 4.

It can be observed that the
expected convergence rates with respect to the time discretization error
are recovered, in the constant $\Delta x,$ constant $C$ or constant $\mu$
convergence studies. It can also be observed that the decoupled variant SL2s, in spite of the loss of second order convergence,
does indeed improve the results with respect to the SL1 method and is competitive with the full second order method
SL2. As a comparison, a standard centered finite difference, second order discretization
in space coupled to an explicit second order method in time yields at the finest resolution an error analogous to  that
of method SL2 but requires approximately three times its CPU time.

In the unstructured case, the quadratic polynomial interpolation naturally associated to $\mathbb{P}_2$ finite
elements was employed and only the SL2s and SL2 methods were considered. The triangular mesh used was chosen with maximum      triangle size $\Delta x$ approximately equal to the corresponding structured meshes.
      The results are reported in Table \ref{tab:test2SL2_unstr}.  While the behaviour of the SL2 scheme is entirely analogous to that of the structured mesh case, the SL2s method shows in this case  little error reduction  
     when the spatial resolution is kept fixed.
  \begin{table}
  \centering
\begin{tabular}{||c|c|c|c|c|c|c||} 
\hline
\multicolumn{3}{||c|}{Resolution} & \multicolumn{2}{c}{Relative error}  & \multicolumn{2}{|c||}{Convergence rates} \\
\hline
$\Delta x$ &  $\lambda$ & $\mu $ & $l_2$ &  $l_{\infty}$ & $p_2$ &  $p_{\infty}$   \\
\hline
$0.04$ &   $16$ & $1.62$ & $0.15$ & $0.16$  & - &  - \\
\hline
$0.04$ &   $8$ & $0.82$ & $7.71\cdot 10^{-2}$ & $8.13\cdot 10^{-2}$  & 0.96  &   0.98  \\
\hline
$0.02$ &   $16$ & $3.2$ & $7.71\cdot 10^{-2}$ & $8.13\cdot 10^{-2}$  & 0.96  &  0.98   \\
\hline
$0.02$ &   $8$ & $1.6$ & $3.92\cdot 10^{-2}$ & $4.13\cdot 10^{-2}$  & 0.97  & 0.97    \\
\hline
\end{tabular}

\bigskip

\begin{tabular}{||c|c|c|c|c|c|c||} 
\hline
\multicolumn{3}{||c|}{Resolution} & \multicolumn{2}{c}{Relative error}  & \multicolumn{2}{|c|}{Convergence rates} \\
\hline
$\Delta x$ &  $\lambda$ & $\mu $ & $l_2$ &  $l_{\infty}$ & $p_2$ &  $p_{\infty}$   \\
\hline
$0.04$ &   $16$ & $1.62$ & $7.65\cdot 10^{-2}$ & $7.95\cdot 10^{-2}$  & - &   -  \\
\hline
$0.04$ &   $8$ & $0.82$ & $3.89\cdot 10^{-2}$ & $4.02\cdot 10^{-2}$  & 0.98  &   0.98  \\
\hline
$0.02$ &   $16$ & $3.2$ & $3.89\cdot 10^{-2}$ & $4.02\cdot 10^{-2}$  & 0.98 &  0.98  \\
\hline
$0.02$ &   $8$ & $1.6$ & $1.96\cdot 10^{-2}$ & $2.02\cdot 10^{-2}$  & 0.98  & 0.99    \\
\hline
\end{tabular}

\bigskip

\begin{tabular}{||c|c|c|c|c|c|c||} 
\hline
\multicolumn{3}{||c|}{Resolution} & \multicolumn{2}{c}{Relative error}  & \multicolumn{2}{|c||}{Convergence rates} \\
\hline
$\Delta x$ &  $\lambda$ & $\mu $ & $l_2$ &  $l_{\infty}$ & $p_2$ &  $p_{\infty}$   \\
\hline
$0.04$ &   $16$ & $1.62$ & $0.11$ & $0.11$  & - & -   \\
\hline
$0.04$ &   $8$ & $0.82$ & $2.88\cdot 10^{-2}$ & $2.66\cdot 10^{-2}$  & 1.93  & 2.05    \\
\hline
$0.02$ &   $16$ & $3.2$ & $2.89\cdot 10^{-2}$ & $2.67\cdot 10^{-2}$  & 1.93 &  2.04  \\
\hline
$0.02$ &   $8$ & $1.6$ & $7.35\cdot 10^{-3}$ & $6.64\cdot 10^{-3}$  & 1.95  &   2.03  \\
\hline 
 \end{tabular}
\caption{Errors for the solid body rotation test, methods SL1 (upper), SL2s (middle) and SL2 (lower) on a structured mesh.}
\label{tab:test2SL2} 
 \end{table} 
 
 \begin{table}
 \centering
\begin{tabular}{||c|c|c|c|c|c|c||} 
\hline
\multicolumn{3}{||c|}{Resolution} & \multicolumn{2}{c}{Relative error}  & \multicolumn{2}{|c||}{Convergence rates} \\
\hline
$\Delta x$ & $\lambda$ &  $\mu $ & $l_2$ &  $l_{\infty}$ & $p_2$ &  $p_{\infty}$   \\
\hline
$0.04$ &  $16$ & $1.62$ & $2.39\cdot 10^{-2}$ & $2.25\cdot 10^{-2}$  & - & -  \\
\hline
$0.04$ &  $8$  & $0.82$ &  $2.72\cdot 10^{-2}$ & $2.84\cdot 10^{-2}$  &  0.19  &  0.34 \\
\hline 
$0.02$ &  $16$  & $3.2$ & $7.20\cdot 10^{-3}$ & $6.32\cdot 10^{-3}$  &  1.73 &  1.83  \\
\hline
$0.02$ &  $8$  & $1.6$ & $2.48\cdot 10^{-3}$ & $2.59\cdot 10^{-3}$  & 3.46  &  3.45 \\
\hline 
 \end{tabular}
 
 \bigskip
 
\begin{tabular}{||c|c|c|c|c|c|c||} 
\hline
\multicolumn{3}{||c|}{Resolution} & \multicolumn{2}{c}{Relative error}  & \multicolumn{2}{|c||}{Convergence rates} \\
\hline
$\Delta x$ &  $\lambda$ & $\mu$ & $l_2$ &  $l_{\infty}$ & $p_2$ &  $p_{\infty}$   \\
\hline
$0.04$ &  $16$ & $1.62$ & $0.129$ & $0.139$  & - & -  \\
\hline
$0.04$ &  $8$ &  $0.82$ & $4.02\cdot 10^{-2}$ & $4.42\cdot 10^{-2}$  & 1.68  &  1.65 \\
\hline 
$0.02$ &  $16$ & $3.2$ & $2.88\cdot 10^{-2}$ & $2.56\cdot 10^{-2}$  & 2.16 & 2.44   \\
\hline
$0.02$ &  $8$ & $1.6$ &  $7.70\cdot 10^{-3}$ & $8.08\cdot 10^{-3}$  & 2.38  & 2.45 \\
\hline 
 \end{tabular}
\caption{Errors for the solid body rotation test, methods SL2s (upper) and SL2 (lower) on an unstructured mesh.}
\label{tab:test2SL2_unstr}
 \end{table}  
%
%
   \subsection{Reaction--diffusion  equations}
 \label{dreq}

 Following \cite{feng:2013}, we consider the Allen--Cahn equation
 $$
 c_t     = \nu \Delta c  -c^3+c
 $$
 on the domain $\Omega=(0, 1)\times (0,1)$, with periodic boundary conditions and
 for $t\in[0,2]$.
 As in \cite{feng:2013}, we take the initial datum
$$
c_0(x,y) = \sin \left( 2 \pi x \right) \sin \left(2 \pi y \right)
$$
and a reference solution is computed by a pseudo-spectral Fourier discretization in space, see e.g. \cite{canuto:2006}, and 
a fourth order Runge--Kutta scheme in time with a very large number of time steps. The results are reported in Table \ref{tab:Test4_1_unstr_b}, for the values $\nu=0.01$ and $\nu=0.05$ of the diffusion parameter, respectively. 
In this case, only unstructured meshes were considered and the reference solution was interpolated onto the unstructured
mesh nodes using a higher order interpolation procedure.
Both tests show a quadratic order of convergence.
 \begin{table}
  \centering
\begin{tabular}{||c|c|c|c|c|c|c||} 
\hline
\multicolumn{3}{||c|}{Resolution} & \multicolumn{2}{c}{Relative error}  & \multicolumn{2}{|c||}{Convergence rates} \\
\hline
$\Delta x$ &  $\Delta t$ & $\mu$ & $l_2$ &  $l_{\infty}$ & $p_2$ &  $p_{\infty}$   \\
\hline
$0.04$ &   $0.1$ & $0.62$ & $1.10\cdot 10^{-3}$ & $1.31\cdot 10^{-3}$  & -  &  -  \\
\hline 
$0.02$ &   $0.05$ & $1.25$ & $2.72\cdot 10^{-4}$ & $2.98\cdot 10^{-4}$  & 2.02 & 2.14   \\
\hline
$0.01$ &   $0.025$ & $2.5$ & $6.53\cdot 10^{-5}$ & $7.06\cdot 10^{-5}$  & 2.06  & 2.08 \\
\hline 
 \end{tabular}

\bigskip

\begin{tabular}{||c|c|c|c|c|c|c||} 
\hline
\multicolumn{3}{||c|}{Resolution} & \multicolumn{2}{c}{Relative error}  & \multicolumn{2}{|c||}{Convergence rates} \\
\hline
$\Delta x$ &  $\Delta t$ & $\mu$ & $l_2$ &  $l_{\infty}$ & $p_2$ &  $p_{\infty}$   \\
\hline
$0.04$ &   $0.1$ & $0.62$ & $2.82\cdot 10^{-2}$ & $4.01\cdot 10^{-2}$  &  -  &  - \\
\hline 
$0.02$ &   $0.05$ & $1.25$ & $7.13\cdot 10^{-3}$ & $8.47\cdot 10^{-3}$  &  1.98 &  2.24  \\
\hline
$0.01$ &   $0.025$ & $2.5$ & $1.97\cdot 10^{-3}$ & $2.20\cdot 10^{-3}$  & 1.86  &  1.94 \\
\hline 
 \end{tabular}
\caption{Error for the Allen--Cahn test with $\nu = 0.01$ (upper) and $\nu = 0.05$ (lower), second order method SL2 on an unstructured mesh.}
\label{tab:Test4_1_unstr_b}
  \end{table}

   \subsection{Advection--diffusion--reaction systems}
 \label{adrsys} 
 We consider in this case a set of four coupled advection--diffusion--reaction equations of the form \eqref{eq:adrsys}
\begin{equation}
\label{adr_exe4}
\frac{\partial c_k}{\partial t}+ u\cdot \nabla c_k - \nu \Delta c_k   = f_k(c_1,\dots,c_4)  \ \ \ k=1,\dots,4
\end{equation}
on the square domain $\Omega=(-5,5)\times(-5,5)$ and on the time interval $t\in [0,5]$.
The advection  field is given by
 coefficients $u=(-\omega y,\omega x), $  $\omega=2\pi/10, $ while the diffusion coefficient is set as $\nu=0.01. $   
 The reaction terms are given by
 \begin{eqnarray}
  f_1&=& (c_1-c_1c_2)-(c_1-c_3)/5 \nonumber \\
   f_2&=& -2(c_2-c_1c_2)-(c_2-c_4)/5  \nonumber \\
    f_3&=& 2(c_3-c_3c_4) \nonumber \\
     f_4&=& -4(c_4-c_3c_4), \nonumber 
 \end{eqnarray}
  which represent two coupled Lotka--Volterra prey-predator systems. As initial datum for $c_1,c_3,$
  the function
  $$
c_0(x,y) = \left \{
\begin{array}{c}
\cos{(2\pi [ (x+2.5)^2+y^2 )] }  \ \ \  {\rm for } \ (x+2.5)^2+y^2\leq \frac 14\\
$ \ $ \\ 
0 \ \ \ \ \ \ \ \  \ \ \ \ \ \ \ \ \ \ \ \ \ \ \ \ \ \ \ \ \ \ \ \ \ {\rm for } \ (x+2.5)^2+y^2> \frac 14
\end{array}
  \right .
$$
was considered, while the initial datum for $c_2,c_4,$ was taken to be equal to $3c_0.$
In this test, only a structured mesh was considered
with constant steps $\Delta x=\Delta y = 1/20.$
A reference solution is computed by a pseudo-spectral Fourier discretization in space and 
a fourth order Runge--Kutta scheme in time, using a very large number of time steps.
The reference solution is reported for two sample components in Figure \ref{fig:refsol_adr}, while
the  absolute error distributions  obtained for the same components 
with the second order method SL2 \eqref{second_order_CN} using cubic interpolation,
using a timestep corresponding to $\lambda\approx 7 $ and  $\mu\approx 1/2, $   
are shown in Figure \ref{fig:err_adr_sl2}.
As a reference, the errors for a second order finite difference approximation of \eqref{adr_exe4}
using a second order Runge--Kutta scheme in time with a time step 20 times smaller are shown in
Figure \ref{fig:err_adr_fd2}, while the errors obtained using a fourth order finite difference approximation
for the advection term in \eqref{adr_exe4} with a third order Runge--Kutta scheme in time are displayed in
Figure \ref{fig:err_adr_fd4}, again computed with a time step 20 times smaller than that used for the SL2 method.
It can be seen that  the SL2 method allows to achieve errors of the same order of magnitude as those
of the third order Runge--Kutta  in time, while allowing for a much larger time step without solving
large algebraic systems. 

\begin{figure}
	\centering
	\includegraphics[width=.45\textwidth]{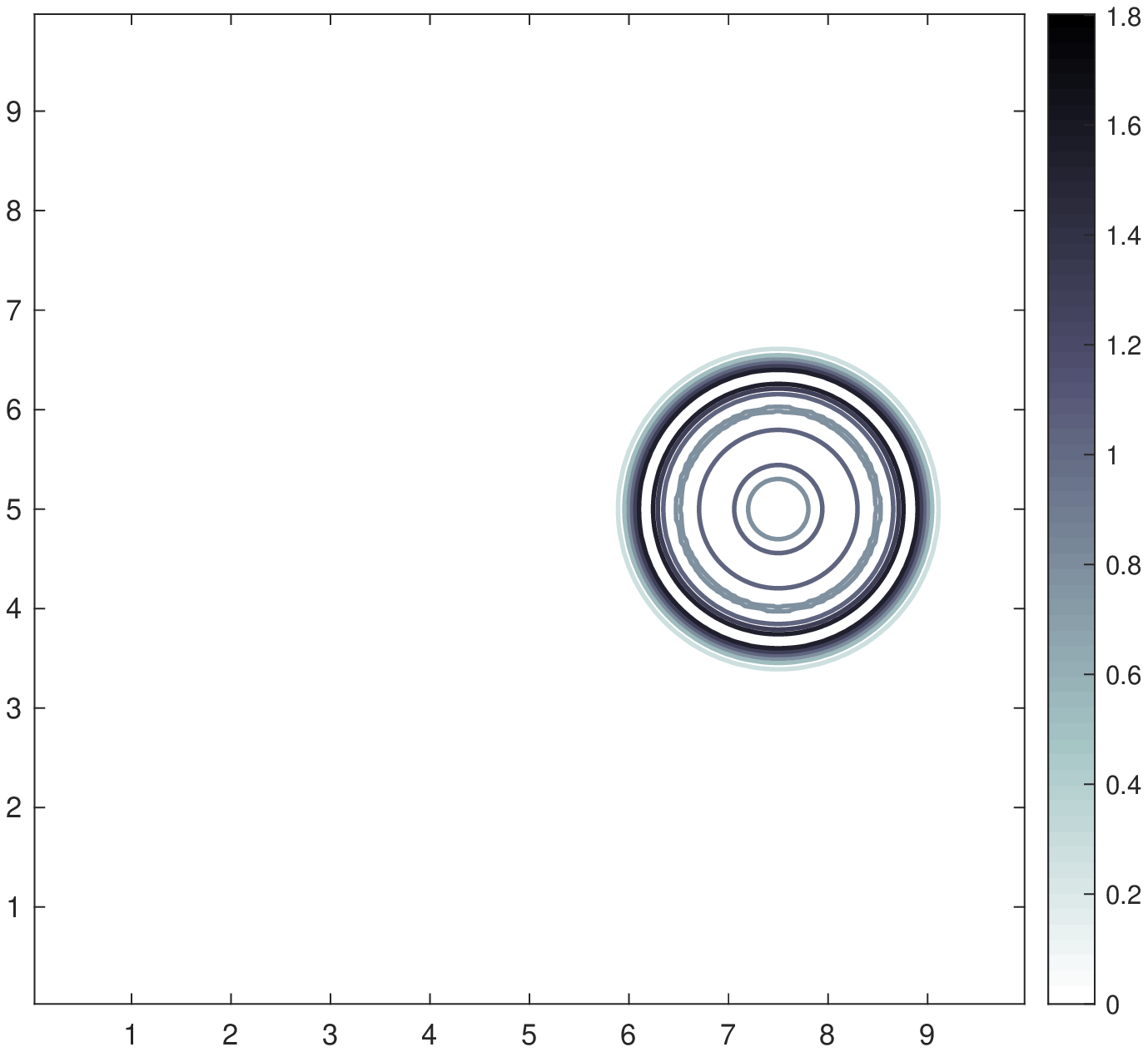} a)
	\includegraphics[width=.45\textwidth]{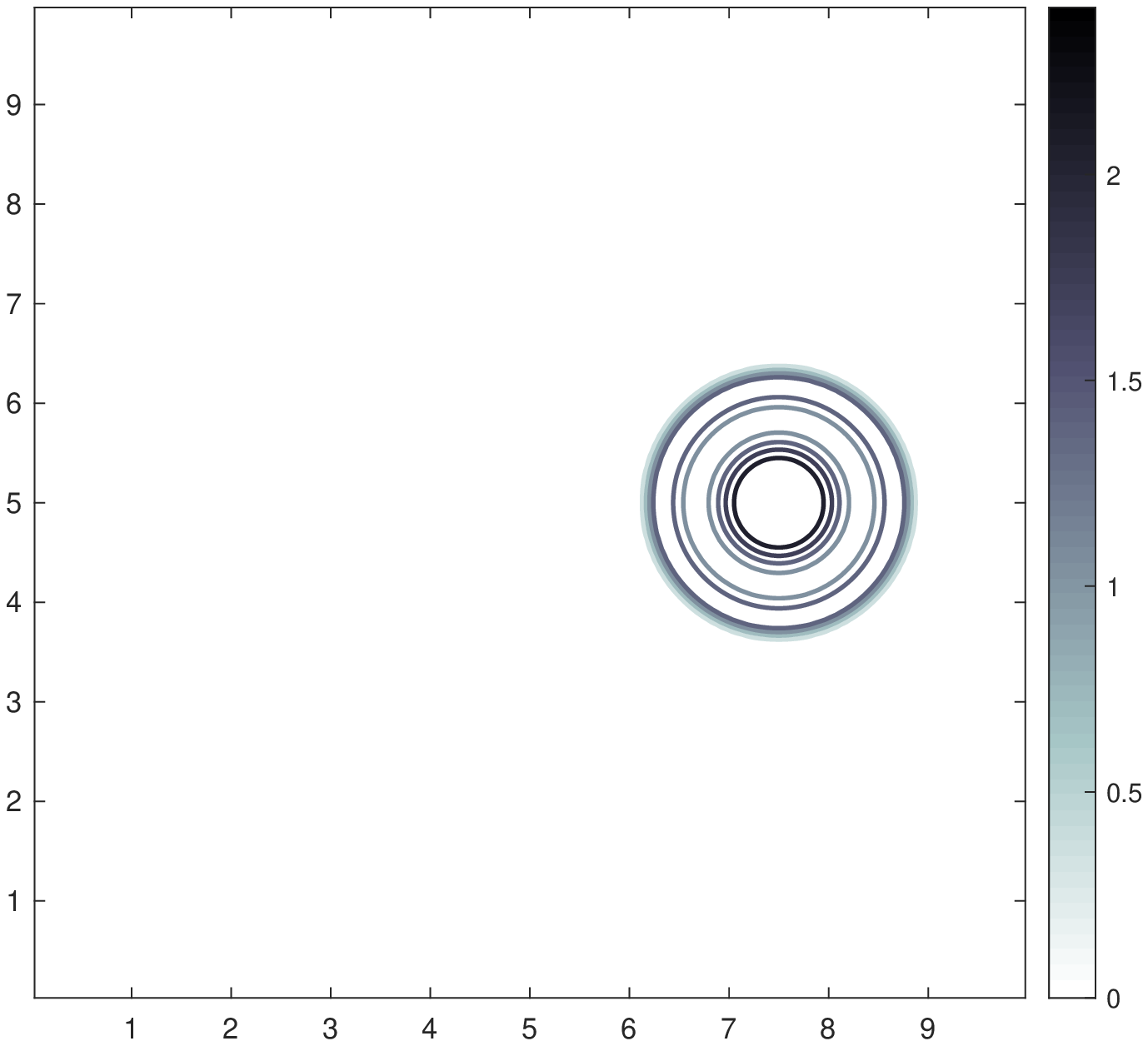} b)
	\caption{Reference solutions for problem \eqref{adr_exe4}, a) component $c_3,$ b) component $c_4$
	at time $T= 5.$}
	\label{fig:refsol_adr}
\end{figure}

\begin{figure}
	\centering
	\includegraphics[width=.45\textwidth]{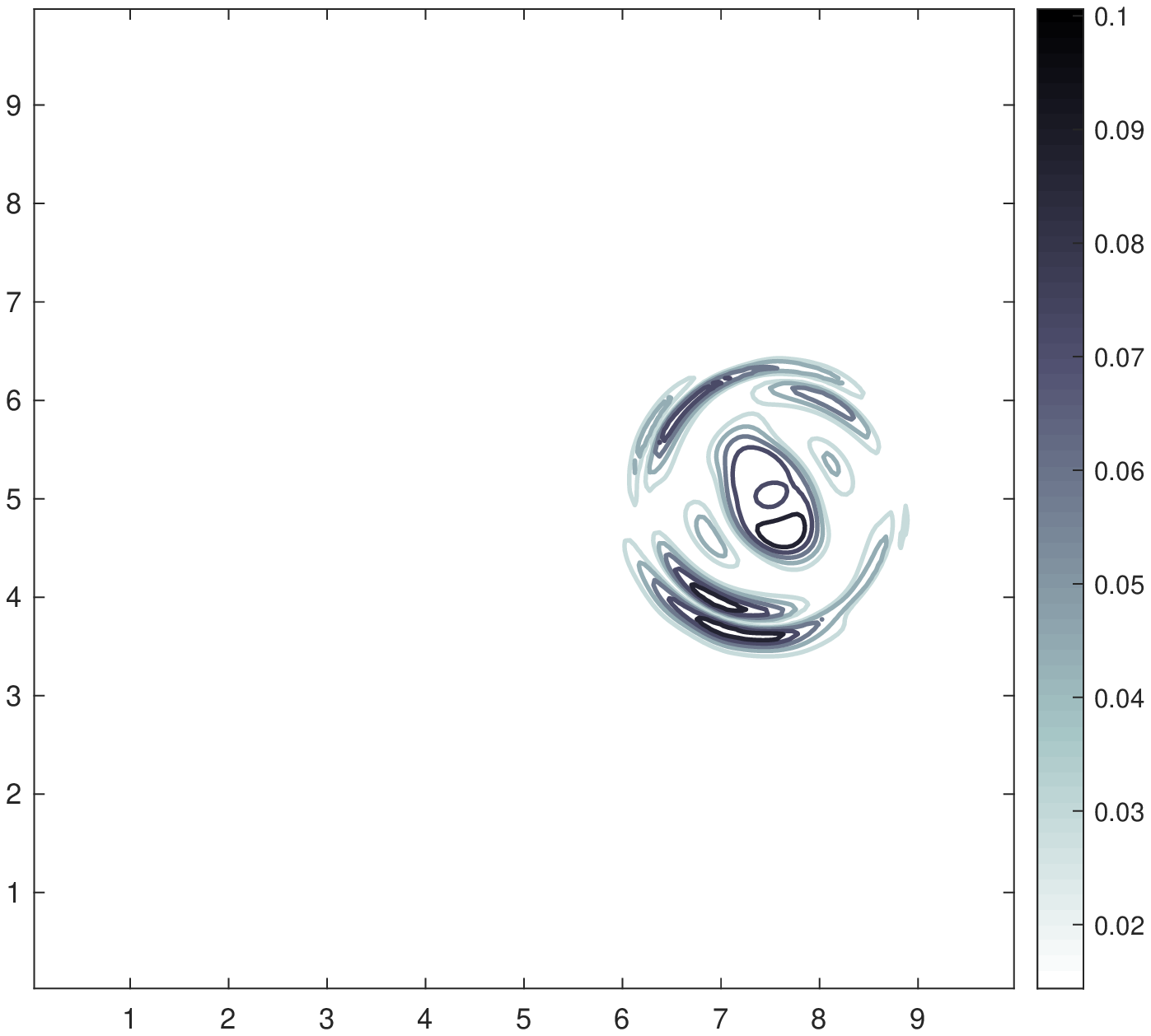} a)
	\includegraphics[width=.45\textwidth]{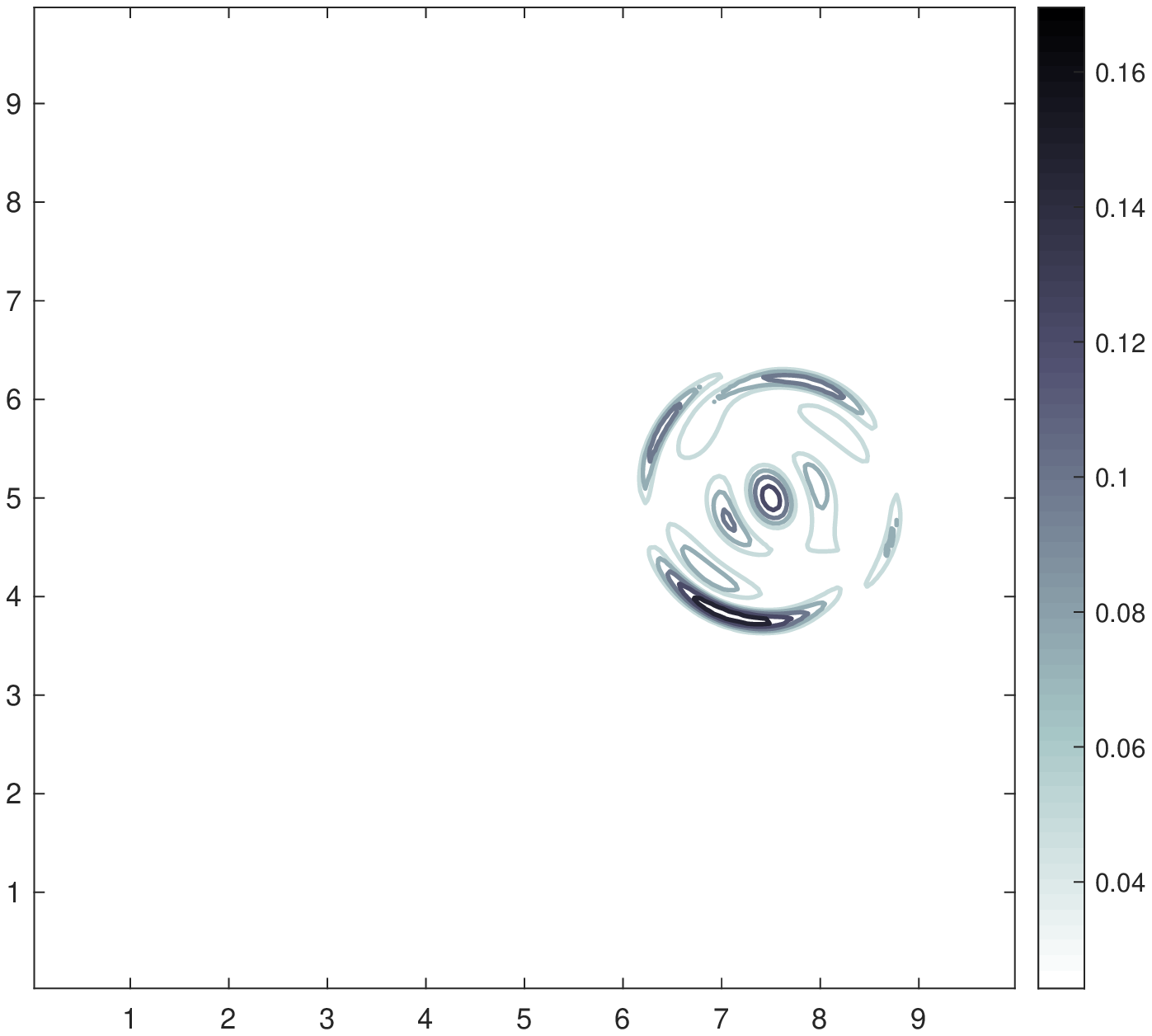} b)
	\caption{Absolute errors of second order SL2 method for problem \eqref{adr_exe4}, a) component $c_3,$ b) component $c_4$
	at time $T= 5.$}
	\label{fig:err_adr_sl2}
\end{figure}

\begin{figure}
	\centering
	\includegraphics[width=.45\textwidth]{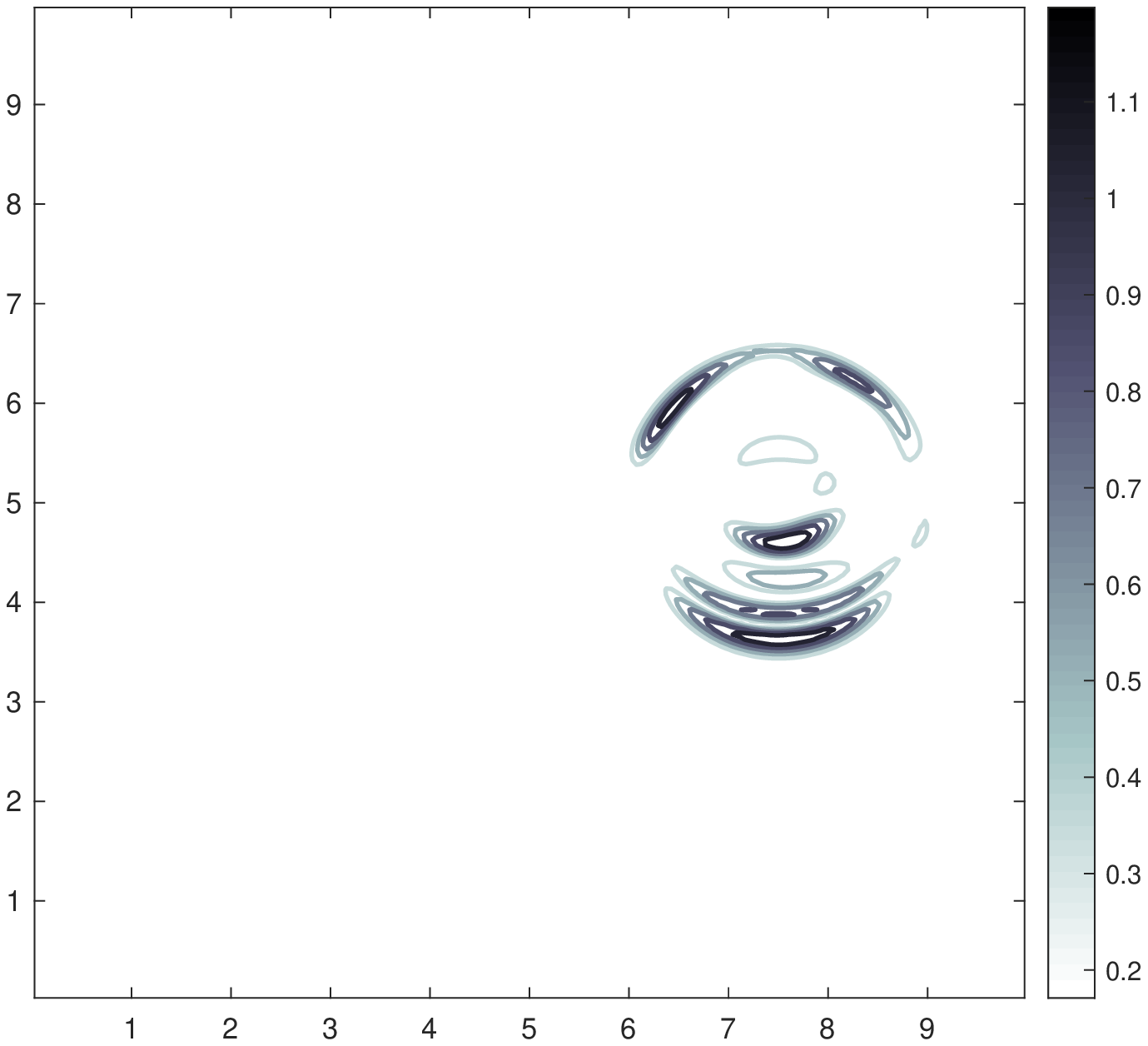} a)
	\includegraphics[width=.45\textwidth]{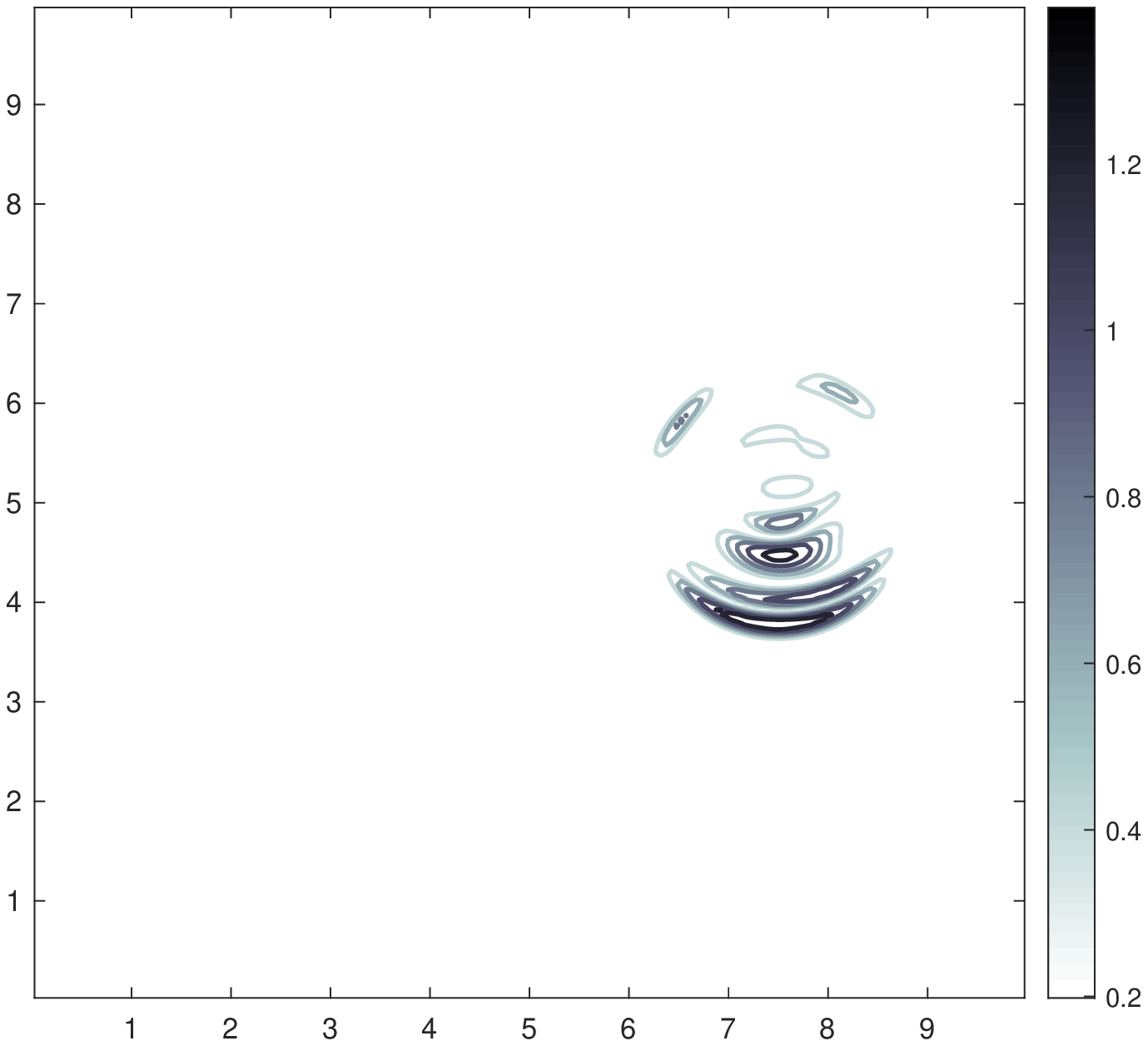} b)
	\caption{Absolute errors of second order finite difference method for problem \eqref{adr_exe4}, a) component $c_3,$ b) component $c_4$
	at time $T= 5.$}
	\label{fig:err_adr_fd2}
\end{figure}

\begin{figure}
	\centering
	\includegraphics[width=.45\textwidth]{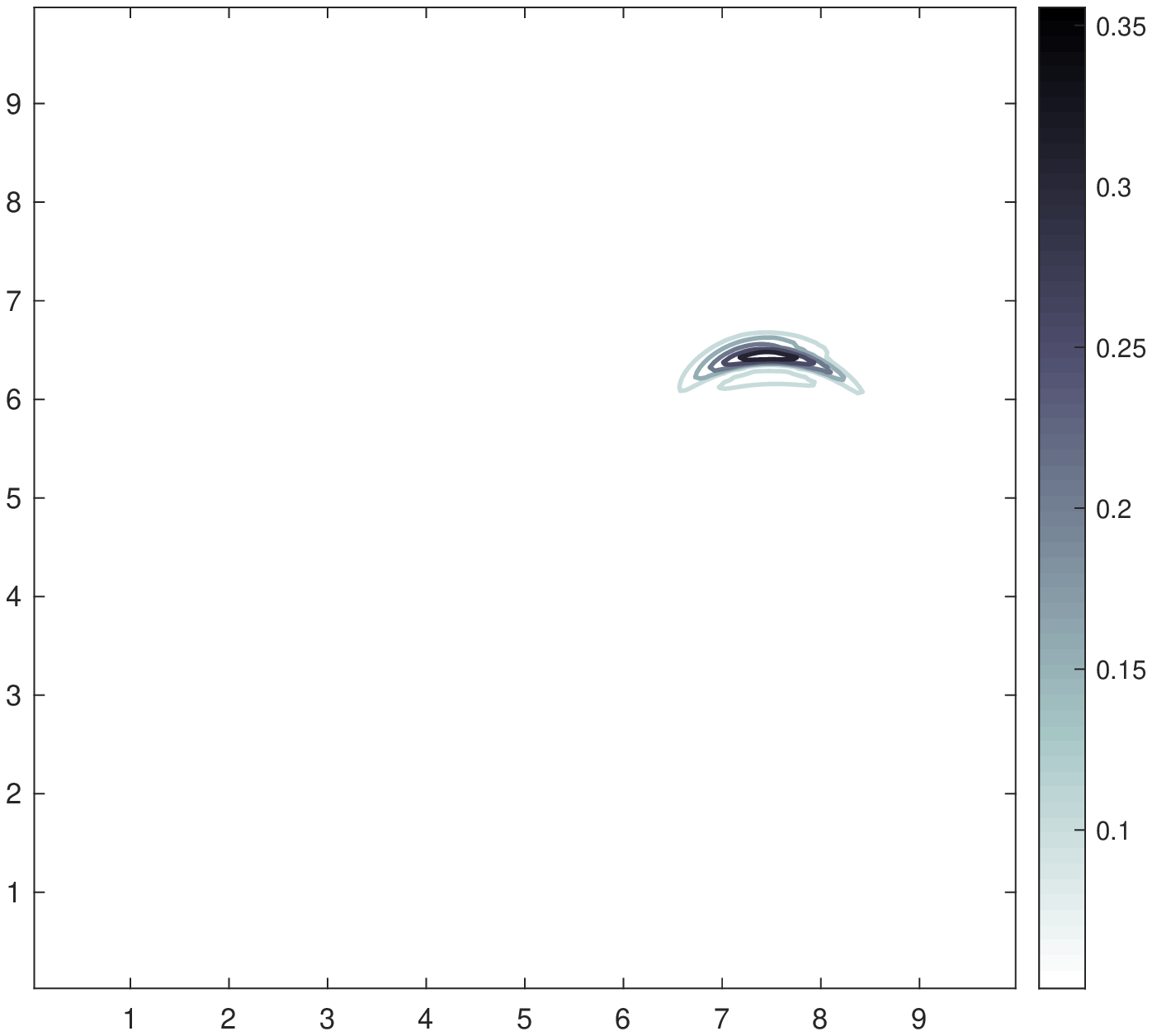} a)
	\includegraphics[width=.45\textwidth]{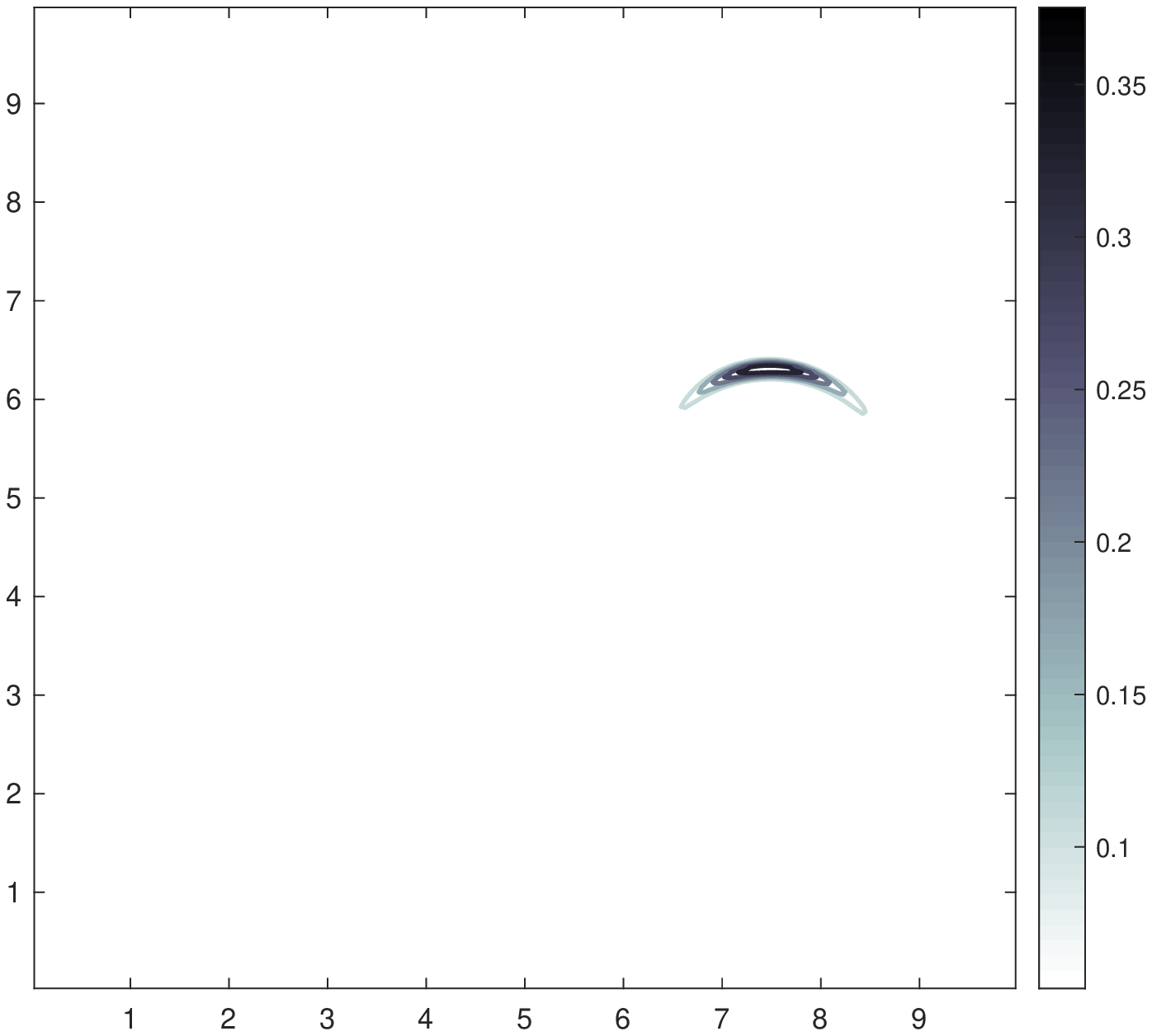} b)
	\caption{Absolute errors of fourth order finite difference method for problem \eqref{adr_exe4}, a) component $c_3,$ b) component $c_4$
	at time $T= 5.$}
	\label{fig:err_adr_fd4}
\end{figure}
%

   \subsection{Advection--diffusion equation, nonhomogeneous boundary conditions}
   \label{test:boundarycondition}
   
In this last set of numerical experiments, we consider nonhomogeneous, possibly time-dependent Dirichlet boundary conditions in four cases: pure diffusion, constant advection--diffusion, solid body rotation with diffusion and advection--diffusion on a nonconvex domain. In all these tests, we have used an unstructured mesh.
In the first three cases, we consider $\Omega=(-1, 1) \times (-1, 1)$, final time $T=1$ and an initial condition in the form of a Gaussian centered at $(x_0,y_0)=(0.5,0)$, with $\sigma= 0.1$. 
In Fig.\ref{fig:boundarymesh}, we show the space meshes $\mathcal{G}_{\Delta x }$ and 
$\mathcal{G}_{h}$ corresponding to the steps $\Delta x=0.04$, $h=0.5, $ which were used to compute the results in the first two rows of Tables \ref{table:diff_dec}-\ref{table:sb1}. In order to have a reference solution to compare with, we compute the exact solution on the whole of $\R^2$ and enforce its values at the boundary as boundary conditions, so that $b(x,y,t)=c(x,y,t)$ for $(x,y)\in\partial\Omega$, $t\in[0,T]$. For all the three cases, we have set $\nu = 0.05$ and $T=1$. In the second and third test, the advection field has been chosen as  $u = (1,0), $ and $u=(-2\pi y,2\pi x), $ respectively.
Tables \ref{table:diff_dec}-\ref{table:sb1} report the numerical errors obtained by the SL2 scheme in these tests, showing in all cases at least a quadratic convergence.

\begin{table}[t!]
  \centering
\begin{tabular}{||c|c|c|c|c|c|c||} 
\hline
\multicolumn{3}{||c|}{Resolution} & \multicolumn{2}{c}{Relative error}  & \multicolumn{2}{|c||}{Convergence rates} \\
\hline
$\Delta x$ &  $\mu$ & $h$ & $l_2$ &  $l_{\infty}$ & $p_2$ &  $p_{\infty}$   \\
\hline
$0.04$ &   $1.56$ & $0.5$ & $4.70\cdot 10^{-3}$ & $1.39\cdot 10^{-2}$  & - & -  \\
\hline
$0.04$ &   $0.78$ & $0.5$ & $3.18\cdot 10^{-3}$ & $1.06\cdot 10^{-2}$  & -  &  -  \\
\hline 
$0.02$ &   $3.12$ & $0.33$ & $3.71\cdot 10^{-4}$ & $1.01\cdot 10^{-3}$  & 3.66 &  3.78  \\
\hline
$0.02$ &   $1.56$ & $0.33$ & $4.35\cdot 10^{-4}$ & $9.57\cdot 10^{-4}$  &  2.87 & 3.47 \\
\hline 
 \end{tabular}
\caption{Errors and convergence rates for the pure diffusion problem with nonhomogeneous Dirichlet conditions, SL2 method, unstructured mesh}
\label{table:diff_dec}
\end{table}
 \begin{table}[t!]
  \centering
\begin{tabular}{||c|c|c|c|c|c|c|c||} 
\hline
\multicolumn{4}{||c|}{Resolution} & \multicolumn{2}{c}{Relative error}  & \multicolumn{2}{|c||}{Convergence rates} \\
\hline
$\Delta x$ &$\lambda$&  $\mu$ & $h$ & $l_2$ &  $l_{\infty}$ & $p_2$ &  $p_{\infty}$   \\
\hline
$0.04$ &$1.25$&   $1.56$ & $0.5$ & $7.35\cdot 10^{-3}$ & $1.18\cdot 10^{-2}$  & - & -  \\
\hline
$0.04$ &$0.625$&   $0.78$ & $0.5$ & $8.35\cdot 10^{-3}$ & $1.32\cdot 10^{-2}$  & -  &  -  \\
\hline 
$0.02$ &$1.25$&   $3.12$ & $0.33$ & $3.76\cdot 10^{-4}$ & $7.59\cdot 10^{-4}$  & 4.29 &  3.96  \\
\hline
$0.02$ &$0.625$ &   $1.56$ & $0.33$ & $2.64\cdot 10^{-4}$ & $5.58\cdot 10^{-4}$  &  4.98 & 4.56 \\
\hline 
 \end{tabular}
\caption{
Errors and convergence rates for the advection--diffusion problem with nonhomogeneous Dirichlet conditions, SL2 method, unstructured mesh.}
\label{table:trasp}
 \end{table}
 \begin{table}[t!]
  \centering
\begin{tabular}{||c|c|c|c|c|c|c|c||} 
\hline
\multicolumn{4}{||c|}{Resolution} & \multicolumn{2}{c}{Relative error}  & \multicolumn{2}{|c||}{Convergence rates} \\
\hline
$\Delta x$ &  $\lambda$ & $\mu$ &$h$ & $l_2$ &  $l_{\infty}$ & $p_2$ &  $p_{\infty}$   \\
\hline
$0.04$ &   $7.85$ & $1.56$ &$0.5$ & $5.62\cdot 10^{-2}$ & $6.09\cdot 10^{-2}$  & - & -  \\
\hline
$0.04$ &   $3.92$ & $0.78$& $0.5$ & $1.49\cdot 10^{-2}$ & $1.60\cdot 10^{-2}$  & -  &  -  \\
\hline 
$0.02$ &   $7.85$ & $3.12$& $0.33$ & $1.49\cdot 10^{-2}$ & $8.98\cdot 10^{-2}$  & 1.91 & -   \\
\hline
$0.02$ &   $3.92$ & $1.56$ & $0.33$ & $3.43\cdot 10^{-3}$ & $3.61\cdot 10^{-3}$  & 2.12  & 2.15 \\
\hline 
 \end{tabular}
\caption{Errors and convergence rates for the solid body rotation problem with nonhomogeneous Dirichlet conditions, SL2 method, unstructured mesh.}
\label{table:sb1}
 \end{table}
 

We finally consider the advection--diffusion equation with $\nu = 0.001$, on the domain $\Omega =([0,1] \times [0, 0.4] )\setminus B_{r_0}\left(x_0, y_0\right)$, where $B_{r_0}\left(x_0, y_0\right)$ denotes a circle with radius $r_0= 0.05$  centered in $\left(x_0, y_0 \right) = \left( 0.1, 0.2 \right)$. The initial datum is $c_0 \left( x, y \right) = 0$
and the boundary condition
\begin{equation*}
b \left( x, y, t \right)= 
\begin{cases}
y \left( 0.4 - y \right) \frac{4}{0.4^2} & \left(x, y \right) \in \left\{ 0 \right\} \times  \left[0, 0.4 \right], t \in \left[0, T \right]\\
1 & \left(x, y \right) \in \partial B_r\left(x_0\right), t \in \left[0, T \right]\\
0  & \text{otherwise.}
\end{cases}
\end{equation*}
The velocity field $u \left(x, y \right)$ is given by
\begin{equation*}
u \left(x, y \right) = \left( u_0 + \frac{u_0 r_0^3}{2 r^3 } - \frac{3 u_0 r_0^3 \left(x - x0 \right)^2}{2 r^5},  -\frac{3 r_0^3 u_0 \left(x - x0 \right) \left(y - y0 \right)}{2 r^5} \right),
\end{equation*}
where we set $u_0 = 0.2$ and $r^2 = \left(x - x_0 \right)^2 + \left(y - y_0 \right)^2$.
In Fig. \ref{fig:mesh_circ_hole}, we show the domain $\Omega$, discretized using a Delaunay mesh $\mathcal G_{\Delta x}$ with $\Delta x=0.1$, refined around the circular hole. In Fig. \ref{fig:T3}, we show the numerical solution computed with SL2 with time step $\Delta t=0.005$ for time $t=0.5, 1, 2, 3$. The nonhomogeneus boundary condition are computed by extrapolation with an extra grid $\mathcal G_h$ with $h=1.5\sqrt{\Delta t}$. In this case, the additional mesh $\mathcal G_h$ has been built around the circular hole, as well as along the external rectangular boundary. Note that, in spite of the discontinuity of the initial configuration and the sharp boundary layer around the hole, the boundary condition is smoothly propagated in the interior of the domain.
\begin{figure}
	\centering
	\includegraphics[width=.8\textwidth]{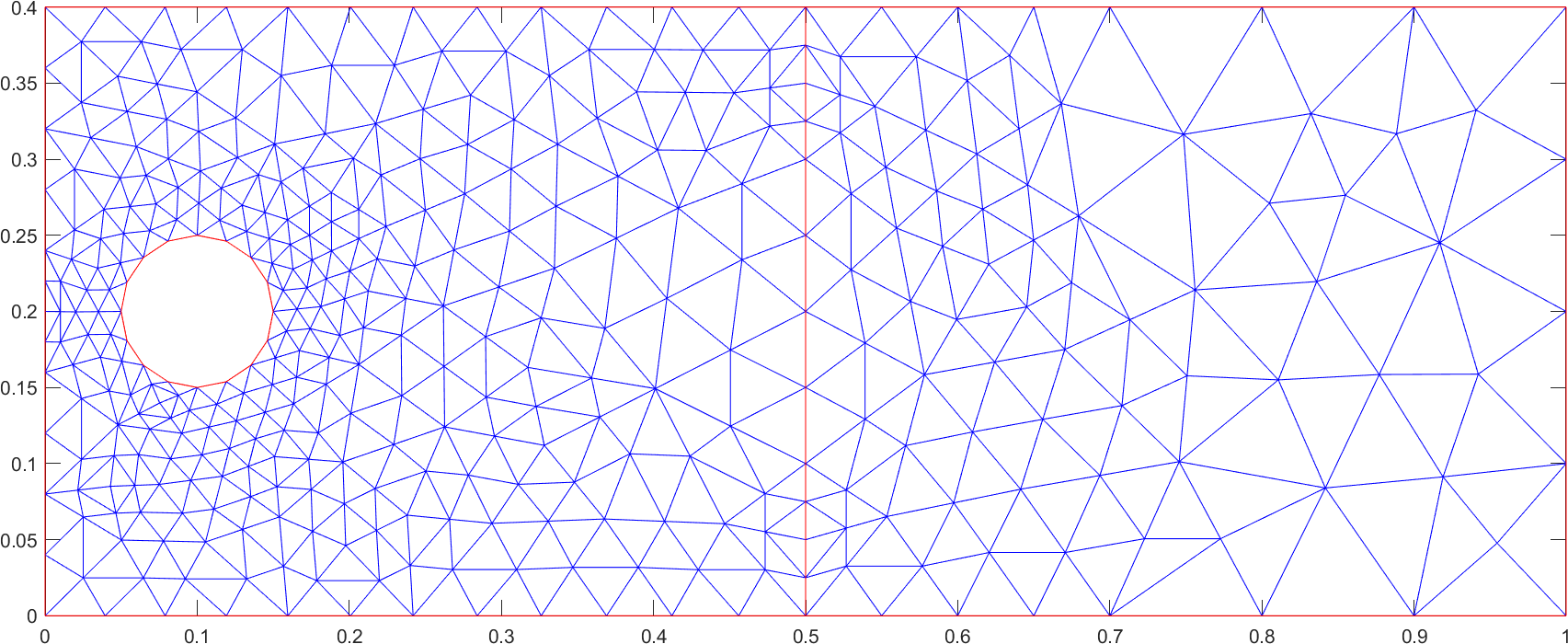}
	\caption{Unstructured mesh for the non-convex problem}
	\label{fig:mesh_circ_hole}
\end{figure}
\begin{figure}
	\includegraphics[width=.9\textwidth]{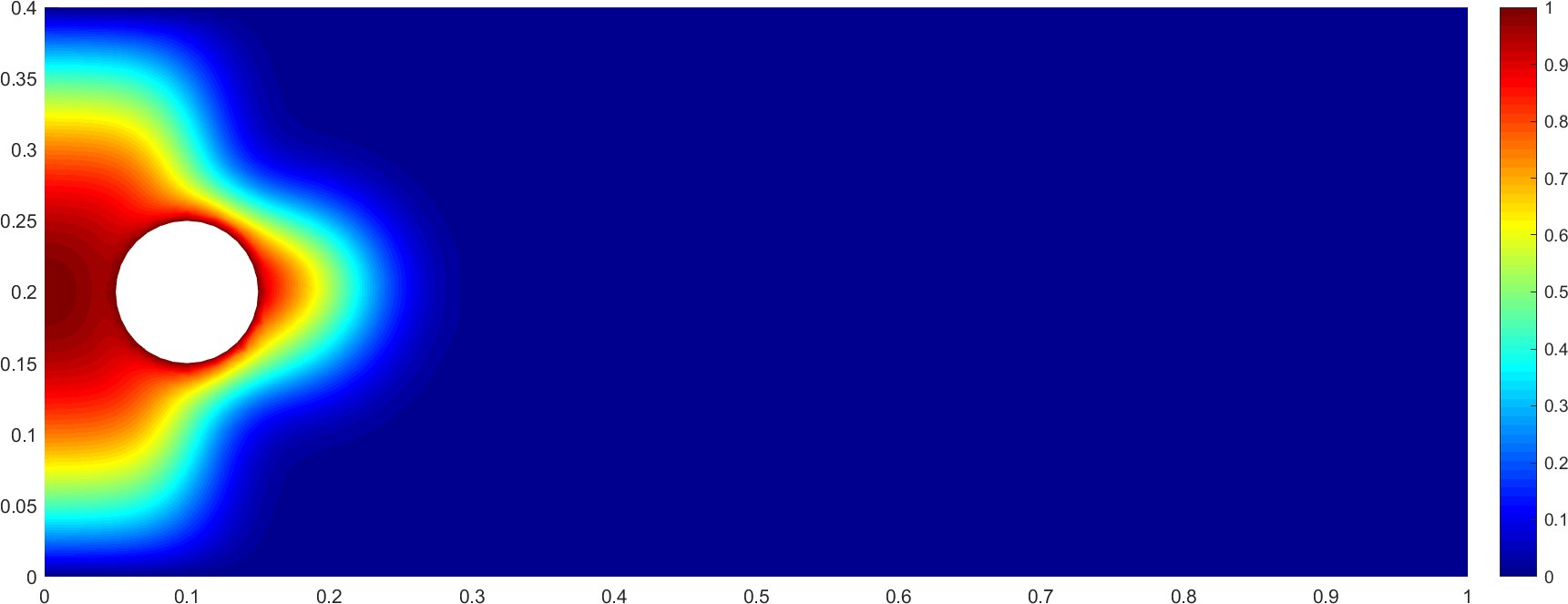}
%
	\includegraphics[width=.9\textwidth]{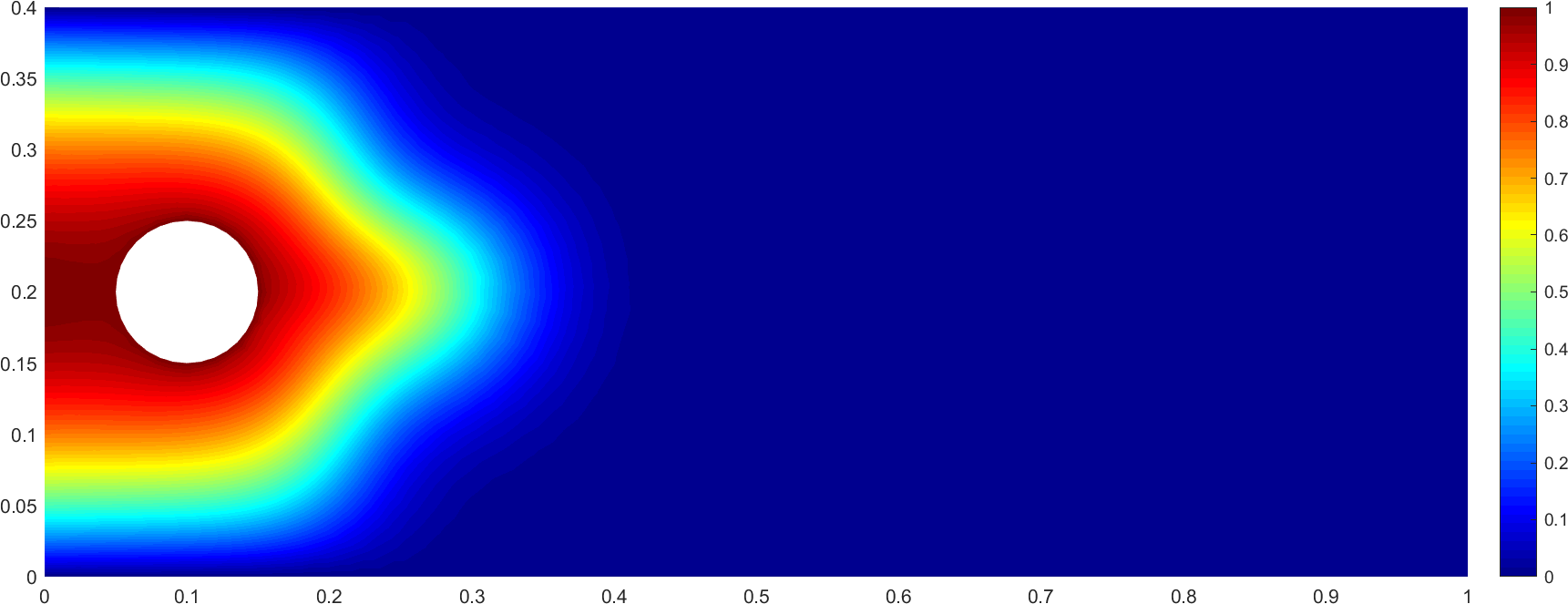}
%
	\includegraphics[width=.9\textwidth]{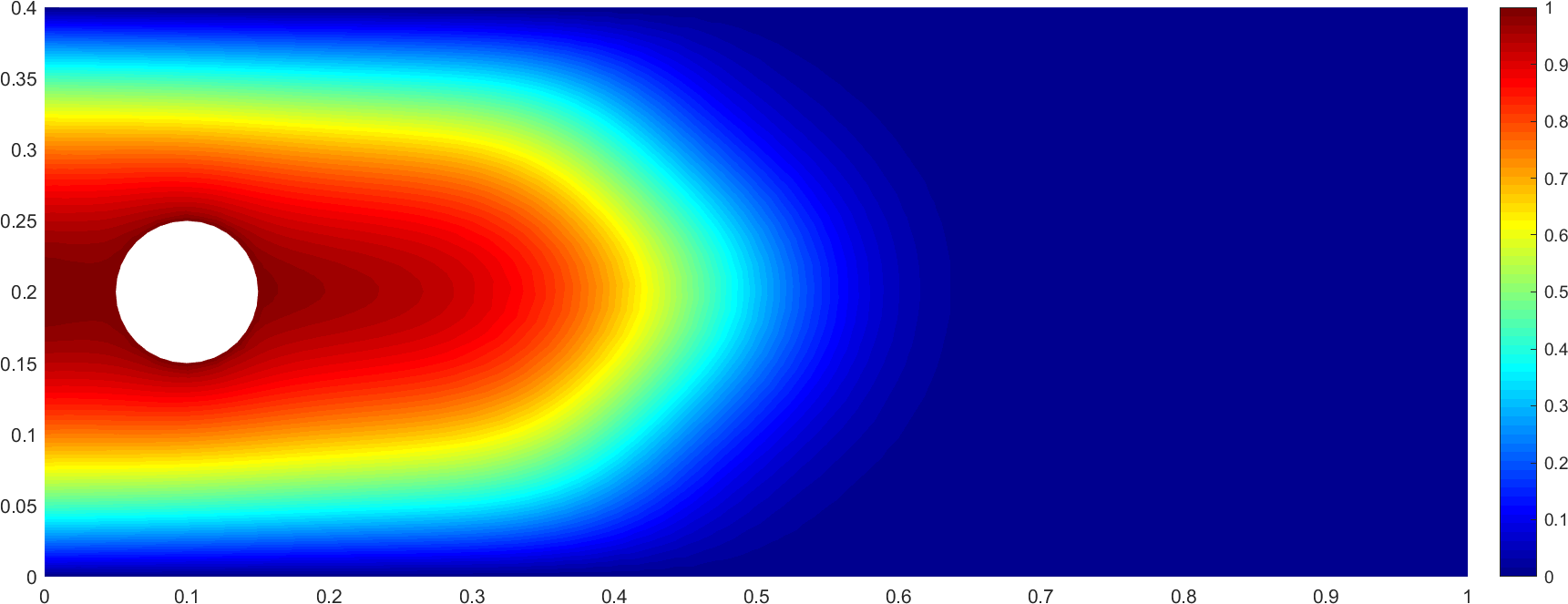}
%
	\includegraphics[width=.9\textwidth]{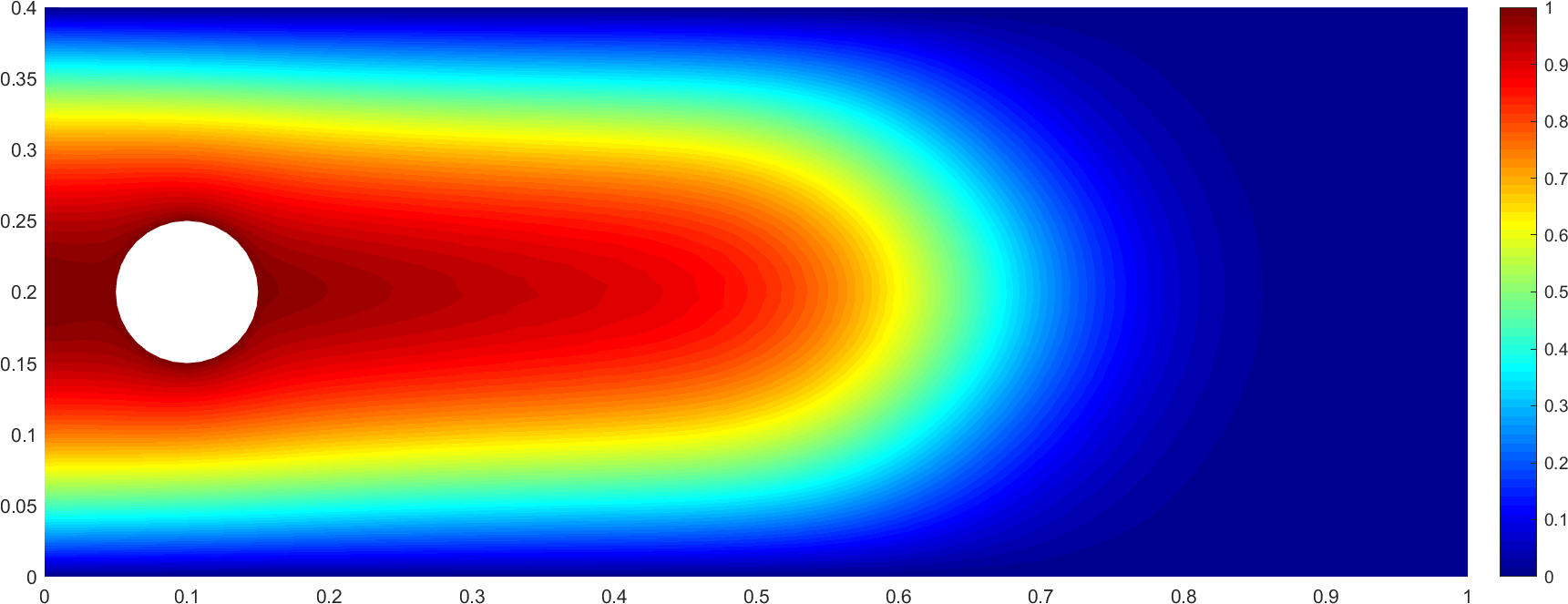}
	\caption{Numerical solution at time $t = 0.5, 1, 2, 3$}
	\label{fig:T3}
\end{figure}

         \section{Conclusions}
 \label{conclu}
 
A family of  fully semi-Lagrangian approaches for the discretization of advection--diffusion--reaction systems has been proposed, which extend the methods outlined in \cite{bonaventura:2014}, \cite{bonaventura:2018} to full second order accuracy.
The stability and convergence
of the basic second order method has been analyzed. 
The proposed methods have been
 validated on a number of classical benchmarks,  on both structured and unstructured meshes.
 Numerical results  show that these methods yield good quantitative agreement with reference
numerical solutions, while being superior in efficiency to standard implicit approaches and to approaches
in which the SL method is only used for the advection term. In future developments, the proposed method
will be extended to higher order discontinuous finite element discretizations along the lines of \cite{tumolo:2015}
and will be applied to the development of second order fully semi-Lagrangian methods for the Navier-Stokes equations
along the lines of  \cite{bonaventura:2020},\cite{bonaventura:2018}. 
\section*{Acknowledgements}
This work has been partly supported by INDAM-GNCS in the framework of the GNCS  2019 project 
\textit{Approssimazione numerica di problemi di natura iperbolica
ed applicazioni}. 
\bibliographystyle{plain}
\bibliography{ns_slag}
\end{document}